\documentclass[a4paper,11pt]{elsarticle}

\usepackage[english]{babel}
\usepackage[latin1]{inputenc}
\usepackage[T1]{fontenc}
\usepackage{amsthm}
\usepackage{amsmath}
\usepackage{amsfonts}
\usepackage{amssymb}

\usepackage{tikzpagenodes}
\usepackage{bigfoot}

\newcommand\bovermat[2]{%
  \makebox[0pt][l]{$\smash{\overbrace{\phantom{%
    \begin{matrix}#2\end{matrix}}}^{\text{#1}}}$}#2}

\def\e{{\epsilon}}

\usepackage{braket}

\def\F{\mathbf F}

\def\e{\mathbf e}

\def\x{\mathbf x}

\def\cC{\mathcal C}

\def\cI{\mathcal I}

\def\cK{\mathcal K}

\def\cM{\mathcal M}

\def\cO{\mathcal O}

\def\cS{\mathcal S}

\def\cV{\mathcal V}

\def\Alt{\mbox{\rm Alt}}
\def\Sym{\mbox{\rm Sym}}
\def\dim{\mbox{\rm dim}}

\def\FF{{\mathbb F}}

\def\f2{{\mathbb F}_{2}}

\newcommand{\AGL}{\mbox{\rm AGL}}

\newcommand{\GL}{\mbox{\rm GL}}

\newcommand{\ga}{\alpha}

\newcommand{\gl}{\lambda}
\newcommand{\gk}{\kappa}
\newcommand{\gr}{\rho}
\newcommand{\gs}{\sigma}

\newcommand{\gt}{\tau}

\newcommand{\ba}{{\bf a}}

\newcommand{\be}{{\bf e}}
\newcommand{\bb}{{\bf b}}

\newcommand{\bv}{{\bf v}}

\newcommand{\bx}{{\bf x}}
\newcommand{\by}{{\bf y}}
\newcommand{\bz}{{\bf z}}

\newcommand{\Span}{{\rm{Span}}}
\def\Fq{\mathbb{F}_{\hspace{-0.7mm}q}}
\def\F2{\mathbb{F}_{\hspace{-0.7mm}2}}

\newtheorem{theorem}{Theorem}[section]

\newtheorem{remark}[theorem]{Remark}

\newtheorem{lemma}[theorem]{Lemma}
\newtheorem{definition}[theorem]{Definition}
\newtheorem{example}[theorem]{Example}

\newtheorem{proposition}[theorem]{Proposition}

\newtheorem{corollary}[theorem]{Corollary}

\newtheorem{algorithm}{Algorithm}[section]

\begin{document}

\begin{frontmatter}

\author[tn]{Carlo Brunetta}
\ead{brunetta@unitn.it}
\author[no]{Marco Calderini \corref{cor1}}
\ead{marco.calderini@uib.no}
\author[tn]{Massimiliano Sala}
\ead{maxsalacodes@gmail.com}

\address[tn]{Department of Mathematics, University of Trento, Via Sommarive 14,
38100 Povo (Trento), Italy}
\address[no]{Department of Informatics, University of Bergen, Thorm{\o}hlensgate 55, 5008 Bergen, Norway}
\cortext[cor1]{Corresponding author}

\title{On Hidden Sums Compatible with A Given Block Cipher Diffusion Layer}

\begin{abstract}
Sometimes it is possible to embed an algebraic trapdoor into a block cipher. Building on previous research, in this paper
we investigate an especially dangerous algebraic structure, which is called a hidden sum and which is related
to some regular subgroups of the affine group. Mixing group theory arguments and cryptographic tools, we pass from
characterizing our hidden sums to designing an efficient algorithm to perform the necessary preprocessing for the
exploitation of the trapdoor.
\end{abstract}


\end{frontmatter}

\section{Introduction}
Sometimes it is possible to embed an algebraic trapdoor into a block cipher. Building on previous research, in this paper
we investigate an algebraic structure, which is called a hidden sum and which is related
to some regular subgroups of the affine group. To be more precise, in \cite{calderini2017elementary}, the authors study some elementary abelian regular subgroups of the affine general linear group acting on a space $V=(\FF_2)^N$, in order to construct a trapdoor for a class of block ciphers.
These subgroups induce alternative operations $\circ$ on $V$, such that $(V,\circ)$ is a vector space over $\FF_2$. In \cite{calderini2017elementary}, it is shown that for a class of these operations, { which we will call \emph{practical hidden sums}}, it is possible to represent the elements with respect to a fixed basis of $(V,\circ)$ in polynomial time. Moreover, an estimate on the number of these operations is given.
Using this class of hidden sums, the authors provide an attack, which works in polynomial time, on ciphers that are vulnerable to this kind of trapdoor.

In this article we continue the analysis started in \cite{calderini2017elementary}. In Section 3 we give a lower bound on the number of practical hidden sums, comparing also this value with a previous upper bound. { From these bounds it is obvious that it is not feasible to generate all possible practical hidden sums due to the large number of these even in small dimensions, e.g. for $N=6$ we have $\sim2^{23}$ practical hidden sums.}
In Section 4 we study the problem of determining the possible maps which are linear with respect to a given practical hidden sums. More precisely, we provide an algorithm that takes as input a given linear map { $\gl$} (with respect to the usual XOR on $V$) and returns { some operations $\circ$, which can be defined over $V$, that are different from the XOR and for which the map $\gl$ is linear, i.e. $\gl(x\circ y)=\gl(x)\circ\gl(y)$ for all $x,\,y$ in $V$.} Our aim is to individuate a family of hidden sums that can weaken the components of a given cipher, or to design a cipher containing the trapdoor based on hidden sums.
In the last section we apply the procedure given in Section 3 to the mixing layer of the cipher PRESENT \cite{CGC-cry-art-Bogdanov2007}, yelding a set of hidden sums which linearize this permutation matrix and might in principle be used to attack the cipher.

\section{Preliminaries and motivations}

We write $\Fq$ to denote the finite field of $q$ elements, where $q$ is a prime power , and $(\Fq)^{s\times t}$ to denote the set of all matrices with entries over $\Fq$ with $s$ rows and $t$ columns. The identity matrix of size $s$ is denoted by $I_s$.
We use 
$$
\e_i=(\underbrace{0,\dots,0}_{i-1},1,\underbrace{0,\dots,0}_{N-i})\in (\Fq)^{N}
$$
to denote the unit vector, which has a $1$ in the $i$th position, and zeros elsewhere. Let $m\ge 1$, the vector (sub)space generated by the vectors $\bv_1,\dots,\bv_m$ is denoted by $\Span\{\bv_1,\dots,\bv_m\}$.

Let $V=(\Fq)^N$, we denote respectively by $\Sym(V)$, $\Alt(V)$ the symmetric and the alternating group acting on $V$.  { We will denote the translation with respect  to a vector $\bv \in V$ by $\gs_\bv:\bx\mapsto \bx+\bv$, and $T(V,+)$ will denote the translations on the vector space $(V,+)$, that is, $T(V,+)=\{\gs_\bv\mid \bv \in V\}$.} By $\AGL(V{ ,+})$ and $\GL(V{ ,+})$ we denote the affine and linear groups of $V$. We write $\langle g_1,\dots,g_m\rangle$ for the group generated by $g_1,\dots,g_m$ in $\Sym(V)$. The map $1_V$ will denote the identity map on $V$.

Let $G$ be a finite group acting  on $V$. We write the action of a permutation $g \in G$ on a vector $\bv \in V$ as $\bv g$. \\
{ We recall that a permutation group $G$ acts \emph{regularly} (or is \emph{regular}) on $V$ if for any pair $\bx,\by\in V$ there exists a unique map $g$ in $G$ such that $\bx g=\by$. Moreover, an {\em elementary abelian group} (or elementary abelian $p$-group) is an abelian group such that any nontrivial element has order $p$. In particular, the group of translations acting on $V$, $T(V,+)$, is a regular elementary abelian group.}

%
%
\subsection{Block ciphers and hidden sums}
Most modern block ciphers are iterated ciphers, i.e. they are obtained by the composition of a finite number $\ell$ of rounds.

%
{ Several classes of iterated block ciphers have been proposed, e.g. \emph{substitution permutation network} \cite{CGC-cry-book-stin95} and \emph{key-alternating block cipher} \cite{CGC-cry-book-daemen2002design}.} Here we present one more recent definition \cite{CGC-cry-art-carantisalaImp} that determines a class large enough to include some common ciphers (AES \cite{daemen2002design}, SERPENT \cite{CGC-cry-art-serpent}, PRESENT \cite{CGC-cry-art-Bogdanov2007}), but with enough algebraic structure to allow for security proofs, { in the contest of symmetric cryptography. In particular, about properties of the groups related to the ciphers}.

Let $V=(\FF_2)^N$ with $N=mb$ and $b\ge 2$. The vector space $V$ is a direct sum
$$
V=V_1\oplus\dots\oplus V_b,
$$
where each $V_i$ has the same dimension $m$ (over $\FF_2$). For any $\bv\in V$, we will write $\bv=\bv_1\oplus\dots \bv_b$, where $\bv_i\in V_i$. 

Any $\gamma\in \Sym(V)$ that acts as $\bv \gamma=\bv_1\gamma_1\oplus\dots\oplus\bv_b\gamma_b$, for some $\gamma_i$'s in $\Sym(V_i)$, is a {\em bricklayer transformation} (a ``parallel map'') and any $\gamma_i$ is a {\em brick}. Traditionally, the maps $\gamma_i$'s are called S-boxes and $\gamma$ a ``parallel S-box''. A linear map $\gl:V\to V$ is traditionally said to be a ``Mixing Layer'' when used in composition with parallel maps. 
For any $I\subset { \{1,\dots,b\}}$, with $I\ne \emptyset, { \{1,\dots,b\}}$, we say that $\bigoplus_{i\in I} V_i$ is a {\em wall}. 
\begin{definition}
A linear map $\gl\in \GL(V{ ,+})$ is a {\em proper mixing layer} if no wall is invariant under $\gl$.
\end{definition}

We can characterize the translation-based class by the following:
\begin{definition}\label{def:tb}
 A block cipher  $\mathcal{C} = \{ \varphi_k \mid k  \in \mathcal{K} \}\subset\Sym(V)$, where $\cK$ is the set containing all the session keys and $\varphi_k$ are keyed permutations, over
  ${\FF_2}$ is called {\em translation based (tb)} if:
       \begin{itemize}
    \item it is the composition of a finite number of $\ell$ rounds, such that any round $\rho_{k,h}$ can be written\footnote{we drop the round indices} as $\gamma\gl\gs_{\bar k}$, where
    \begin{itemize}
    \item[-] $\gamma$ is a round-dependent bricklayer transformation (but it does not depend on $k$),
     \item[-] $\lambda$ is a round-dependent linear map (but it does not depend on $k$),
    \item[-] $\bar{k}$ is in $V$ and depends on both $k$ and the round ($\bar{k}$ is called a ``round key''),
     \end{itemize}
     
      \item for at least one round, which we call \emph{proper}, we have (at the same time) that $\gl$ is proper and that the map $\mathcal{K} \to V$ given by $k \mapsto \bar{k}$ is
      surjective.
    \end{itemize}

\end{definition}
For a tb cipher it is possible to define the following groups.
For each round $h$
$$
\Gamma_h(\cC) = \langle \gr_{k,h} \mid k\in \cK \rangle \subseteq \Sym(V),
$$
and the round function group is given by
$$
\Gamma_\infty(\cC)=\langle \Gamma_h(\cC) \mid h=1,\dots,\ell\rangle.
$$

An interesting problem is determining the properties of the permutation group $\Gamma_{\infty}(\mathcal{C})=\Gamma_{\infty}$  that imply  weaknesses of the cipher. 
A {\em trapdoor} { (sometimes called \textit{backdoor} see \cite{filion})} is a hidden structure of the cipher, whose knowledge allows an attacker to obtain information on the key or to decrypt certain ciphertexts.

The first paper dealing with properties of $\Gamma_{\infty}$ was published by Paterson \cite{CGC-cry-art-paterson1}, who showed  that if this group is imprimitive, then it is possible to embed a trapdoor into the cipher. 
On the other hand, if the group is primitive no such trapdoor can be inserted. 
{Other works, dealing with security properties of ciphers related to groups theory, study also whenever the group generated by the round functions is large \cite{CGC-cry-art-sparr08,CGC-cry-art-Wern2}. However, as shown in \cite{court}, even if a given set of round functions generates a large permutation group, it might be possible to approximate these round function by another set of round functions which generates a small group.}

{Similarly, the primitivity of $\Gamma_{\infty}$ does not guarantee the absence of { other types of} trapdoors {based on the group structure of $\Gamma_{\infty}$}.} For example, if the group
is contained in $\AGL(V)$ {(which is a primitive group)}, the encryption function is affine and  once we know the image of a basis of $V$ and the image of the zero vector, then we are able to reconstruct the matrix and the translation that compose the map. This motivated the authors of \cite{calderini2017elementary} on studying different abelian operations, which can be individuate on $V$ in order to embed a trapdoor into a block cipher. { The authors in \cite{calderini2017elementary} called these operations {hidden sums}.}\\

In the remainder of this section, we summarize the theory presented in \cite[Section 2-3-5]{calderini2017elementary}.
\begin{remark}
If $T\subset \Sym(V)$ is an elementary abelian regular group { acting on $V$}, { then} there exists a vector space structure $(V,\circ)$ such that $T$ is the related translation group. In fact, since $T$ is regular the elements of the group can be labelled $$T=\{\gt_\ba\mid\ba\in V\},$$ where $\gt_\ba$ is the unique map in $T$ such that $0\mapsto \ba$. Then, the sum between two elements is defined by $\bx\circ\ba:=\bx\gt_\ba$. Clearly, $(V,\circ)$ is an abelian additive group and thus a vector space over $\FF_2$.\\
{Conversely, if $(V,\circ)$ is a vector space over $\f2$, then its translation group, given by the maps $\gt_\ba:\bx\mapsto\bx\circ\ba$, is an elementary abelian group acting regularly on $V$.}
\end{remark}


In the following, with the symbol $+$ we refer to the usual sum over the vector space $V$. We denote by $T_+=\mathrm{T}(V,+)$, $\AGL(V,+)$ and $\GL(V,+)$, respectively, the translation, affine and linear groups w.r.t. $+$. We use $T_\circ$, $\AGL(V,\circ)$ and $\GL(V,\circ)$ to denote, respectively, the translation, affine and linear groups corresponding to a hidden sum $\circ$. { That is, $T_\circ=\{\gt_\ba\mid\ba\in V\}$ where $\gt_\ba:\bx\mapsto\bx\circ\ba$, $\GL(V,\circ)$ is the group of the maps $\gl$ such that $(\bx \circ \by)\gl=\bx \gl\circ\by\gl$ for all $\bx,\by\in V$, and any map in $\AGL(V,\circ)$ is the composition of a map $\gl\in\GL(V,\circ)$ and a map  $\gt_\ba\in T_\circ$.\\
Since we will focus on a particular class of these operation $\circ$, we define a \emph{hidden sum} any vector space structure $(V,\circ)$, different from the usual $(V,+)$, such that $T_+\subseteq \AGL(V,\circ)$. While, a {\em practical hidden sum} is a hidden sum such that $T_\circ$ is also contained in $\AGL(V,+)$}.

From \cite{calderini2017elementary} we have three interesting problems: 
\begin{enumerate}
\item Determine the operations $\circ$ (equivalently the translation groups $T_\circ$) such that $T_+\subseteq \AGL(V,\circ)$.
\item Determine the operations $\circ$  such that $T_\circ\subseteq \AGL(V,+)$ (i.e. those which are practical hidden sums).
\item Given a parallel S-box $\gamma$ and a mixing layer $\gl$, determine the operations $\circ$ such that $\gamma,\,\gl\in  \AGL(V,\circ)$ or $\gamma\gl\in  \AGL(V,\circ)$.\label{prob3}
\end{enumerate}
{ Problem 1 is related to identifying the class of hidden sums that could potentially contain $\Gamma_\infty(\cC)$, of a given cipher $\cC$, in the group $\AGL(V,\circ)$ since $T_+\subseteq\Gamma_\infty(\cC)$; or at least to individuating hidden sums for which the XOR with a round key is affine with respect to the operation $\circ$. \\
Operations with the characteristic given in Problem 2 and that of Problem 1 permit to represent an element $\bv$ in $(V,\circ)$ efficiently, as we will see in Algorithm \ref{alg:comb}. \\
The last problem permits to understand if a given block cipher could be modified to introduce the algebraic structure of a hidden sum.}\\

The following vector space plays an important role for studying these problems. Let $T$ be any subgroup of the affine group, we can define the vector space 
$$
U(T)=\{\bv\in V\mid \gs_\bv\in T\}.
$$
In \cite{calderini2017elementary} the authors proved
\begin{proposition}[Proposition 3.6 \cite{calderini2017elementary}]\label{prop:intmax}
Let $V=(\FF_2)^N$, $\dim(V)=N$. Let $T\subseteq \AGL(V,+)$ be an elementary abelian regular subgroup. If $T\ne T_+$, then $1\le \dim(U(T))\le N-2$.
\end{proposition}

A characterization is given in \cite{calderini2017elementary} for the maps that generate a translation group $T_\circ\subseteq\AGL(V,+)$ such that $T_+\subseteq \AGL(V,\circ)$.
We recall that for every $\ba$, $\gt_\ba\in T_\circ\subset \AGL(V,+)$ can be written as $\gk_\ba\gs_\ba$ for a linear map $\gk_\ba\in\GL(V,+)$. We will denote by $\Omega(T_\circ)=\{\gk_\ba\mid \ba\in V\}\subset\GL(V,+)$. Moreover $\gk_\ba=1_V$ if and only if $\ba \in U(T_\circ)$.

{ We recall the following definition.
\begin{definition}
An element $r$ of a ring $R$ is called {nilpotent} if $r^m=0$ for some $m\ge 1$ and it is called {unipotent} if $r-1$ is nilpotent, i.e. $(r-1)^m=0$ for some $m\ge 1$. 

Let $G\subseteq \GL(V,+)$ be a subgroup consisting of unipotent permutations, then $G$ is called unipotent.
\end{definition}}
 
\begin{lemma}[Lemma 3.11 \cite{calderini2017elementary}]
Let $T_\circ\subseteq \AGL(V,+)$. We have that $\Omega(T_\circ)$ is unipotent. Moreover, for all $\by\in V$ $\gk_\by^2=1_V$.
\end{lemma}

\begin{theorem}[Theorem 3.17 \cite{calderini2017elementary}]\label{th:forma}
{ Let $N=n+d$ and $V=(\FF_2)^{N}$, with $n\ge 2$ and $d\ge 1$}. Let $T_\circ\subseteq \AGL(V,+)$ be such that $U(T_\circ)=\Span\{\be_{n+1},\dots,\be_{n+d}\}$. Then, $T_+\subseteq \AGL(V,\circ)$ if and only if for all $\gk_\by\in \Omega(T_\circ)$ there exists a matrix $B_\by\in (\FF_2)^{n\times d}$ such that
$$
\gk_\by=\left[\begin{array}{cc}
I_{n} & B_\by\\
0 & I_{d} \end{array}\right].
$$
\end{theorem}
Note that we can always suppose that  $U(T_\circ)$ is generated by the last vectors of the canonical basis, as any group $T_\circ$ is conjugated to a group $T_{\circ'}$ such that $U(T_{\circ'})=\Span\{\be_{n+1},\dots,\be_{n+d}\}$ (see \cite[Theorem 3.14]{calderini2017elementary}). \\
{ Indeed, let $T_\circ$ be a translation group of a practical hidden sum, with $U(T_\circ)=\Span\{\bv_1,\dots,\bv_d\}$. We have that $\bv_1,\dots,\bv_d$ are linear independent with respect to the sum $+$, which implies that there exists a linear map $g\in GL(V,+)$ such that $\bv_ig=\be_{n+i}$ for $i=1,\dots,d$. Thus, the conjugated group $T_{\circ'}=gT_\circ g^{-1}$ is such that $U(T_{\circ'})=\Span\{\be_{n+1},\dots,\be_{n+d}\}$ and $\circ'$ is again a practical hidden sum.\\
Practical hidden sums with $U(T_{\circ})=\Span\{\be_{n+1},\dots,\be_{n+d}\}$ are easier to study thanks to the particular structure reported in Theorem \ref{th:forma}.} 

\begin{remark}\label{rm:basis}
When $U(T_\circ)$ is generated by the last vectors of the canonical basis, then the maps $\gt_{\be_i}$ generate $T_\circ$, i.e. the canonical vectors form a basis also for the vector space $(V,\circ)$. { Indeed, thanks to the form of the maps $\gk_{\be_i}$ given in Theorem \ref{th:forma}, it is possible to verify that combining the maps $\gt_{\be_i}$'s we are able to obtain all the $2^N$ different maps contained in $T_\circ$ (see \cite[Lemma 3.7 and Corollary 3.8] {calderini2017elementary} for further details).}
\end{remark}

In \cite{calderini2017elementary} it turns out that for any practical hidden sum $T_\circ\subseteq \AGL(V,+)$, any given vector $\bv$ can be represented in $(V,\circ)$, with respect to the canonical vectors, in polynomial time ($\cO(N^3)$). 
{ Let $T_\circ$ be a practical hidden sum with $U(T_{\circ})=\Span\{\be_{n+1},\dots,\be_{n+d}\}$ ($\dim(V)=N={n+d}$), then the algorithm for determining the representation of a given vector $\bv$ is the following
\begin{algorithm}\label{alg:comb}
\ \\
{\bf Input:} vector $\bv=(v_1,\dots,v_{N})\in V$\\
{\bf Output:} coefficients $\ga_1,\dots,\ga_{N}$ such that $\ga_1\be_1\circ...\circ\ga_{N}\be_{N}=\bv$.\\
$[i]$ $\ga_i\leftarrow v_i$ for $1\le i\le n$;\\
$[ii]$ $\bv'\leftarrow \bv\gt_{\be_1}^{\ga_1}\cdots\gt_{\be_{n}}^{\ga_{n}}$;\\
$[iii]$ $\ga_i\leftarrow v_i'$ for $n+1\le i\le n+d$;\\
return $\ga_1,\dots,\ga_{N}$.
\end{algorithm}
\noindent{{\em Correctness of Algorithm \ref{alg:comb}:}\\
To find the coefficients $\ga_1,\dots,\ga_{N}$ is equivalent to individuating the maps $\gt_{\be_i}$ which generate the map $\gt_\bv$. So, we need to understand which maps (among  the $\gt_{\be_i}$)  are needed to send $0$ in $\bv$, or equivalently $\bv$ in $0$ (since $T_\circ$ is an elementary regular subgroups, $\gt_\bv$ is the unique map such that $0\gt_\bv=\bv$ and  $\bv\gt_\bv=0$).

Note that thanks to the form of the maps $\gk_{\be_i}$ given in Theorem \ref{th:forma}, whenever we apply a map $\gk_{\be_i}$ to a vector $\bx$, the first $n$ entries of $\bx$ are left unchanged. So, if the entry $v_i$ is equal to $1$ for $1\le i\le n$, to delete it we need to apply $\gt_{\be_i}$. This explains step [ii]. 

Now, $\bv'=\bv\gt_{\be_1}^{\ga_1}\cdots\gt_{\be_{n}}^{\ga_{n}}$ is such that the first $n$ entries are all zero. So, we need to delete the entry $v'_j$ for $n+1\le j\le N$, whenever $v'_j=1$. Since $\gt_{\be_j}=\gs_{\be_j}$ for $n+1\le j\le N$, we apply $\gt_{\be_j}$ if $v'_j=1$.

So, we have obtained a map $\gt\in T_\circ$ such that $\bv\gt=0$, and since this is unique we have $\gt=\gt_\bv$.}\\

Note that if $U(T_{\circ})$ is not generated by the last vectors of the canonical basis, then we can consider the conjugated group $T_{\circ'}=gT_\circ g^{-1}$, with $g$ as above, and we need to apply Algorithm \ref{alg:comb} to $\bv g$, obtaining its representation in $(V,\circ')$. This leads us to obtain the representation of $\bv$ in $(V,\circ)$ with respect to the basis $\{\be_1 g^{-1},...,\be_N g^{-1}\}$ of $(V,\circ)$. In particular, the map $g$ is an isomorphism of vector space between $(V,\circ)$ and $(V,\circ')$.}

From this, the authors proved that a hidden sum trapdoor is practical, whenever $\Gamma_\infty\subseteq\AGL(V,\circ)$ and $T_\circ\subseteq \AGL(V,+)$.\\
{A class of these practical hidden sums is used in \cite{Rob} to weaken the non-linearity of some APN S-boxes. In addition, in  \cite{Rob} a differential attack with respect to hidden sums is presented.}

An upper bound on the number of some hidden sums is given in \cite{calderini2017elementary} and reported below.\\
Let $n\ge 2$ and $d\ge 1$. { According to Theorem \ref{th:forma}, denote the matrix} $\gk_{\be_i}$ by
\begin{equation}\label{eq:gammai}
\gk_{\be_i}=\left[\begin{array}{cccc}
 & b^{(i)}_{1,1}&\dots&b^{(i)}_{1,d}\\
 I_{n}  & \vdots& &\vdots\\

 & b^{(i)}_{n,1}&\dots&b^{(i)}_{n,d}\\
0 & &I_{d}& \end{array}\right],
\end{equation}
for each $1\le i\le n$.
\begin{theorem}[Theorem 5.5 \cite{calderini2017elementary}]\label{lm:numero}
Let $N=n+d$ and $V=(\FF_2)^{N}$, with $n\ge 2$ and $d\ge 1$. The number of elementary abelian regular subgroups $T_\circ\subseteq \AGL(V,+)$ such that $\dim (U(T_\circ))=d$ and $T_+\subseteq \AGL(V,\circ)$ is
\begin{equation}\label{eq:numero}
{N\brack d}_2\cdot |\cV(\cI_{n,d})|
\end{equation}
where  $\cV(\cI_{n,d})$ is the variety of the ideal $\cI_{n,d}\subset\FF_2\left[b^{(s)}_{i,j}\mid {1\le i,s\le n,\, 1\le j\le d}\right]$ generated by 
$$
\cS_0\cup\cS_1\cup\cS_2\cup\cS_3
$$
with
$$
\begin{aligned}
\cS_0&=\left\{\left(b^{(s)}_{i,j}\right)^2-b^{(s)}_{i,j}\mid {1\le i,s\le n,\, 1\le j\le d}\right\}\\
\cS_1&=\left\{\prod_{i=1}^{n}\prod_{j=1}^{d}\left(1+\sum_{s\in S}b^{(s)}_{i,j}\right)\mid S\subseteq \{1,\dots,n\},S\ne \emptyset\right\},\\
\cS_2&=\left\{b^{(s)}_{i,j}-b^{(i)}_{s,j}\mid {1\le i,s\le n,\, 1\le j\le d}\right\},\\
\cS_3&=\left\{b^{(i)}_{i,j}\mid {1\le i\le n,\, 1\le j\le d}\right\},
\end{aligned}
$$
and ${N\brack d}_q=\prod_{i=0}^{d-1}\frac{q^{N-i}-1}{q^{d-i}-1}$ is the Gaussian Binomial.
\end{theorem}

\begin{proposition}[Proposition 5.6 \cite{calderini2017elementary}]\label{prop:upperbound}
Let $\cI_{n,d}$ defined as in Theorem \ref{lm:numero}, then
$$
 |\cV(\cI_{n,d})|\le 2^{d\frac{n(n-1)}{2}}-1-\sum_{r=1}^{n-2}{n\choose r}\prod_{1}^{\binom{n-r}{2}}\left(2^{d}-1\right)=\mu(n,d).
$$
\end{proposition}

\section{New lower bounds and asymptotic estimates}

In this section we will provide a lower bound on the cardinality of the variety $\cV(\cI_{n,d})$. Moreover we will show that the ratio between the upper bound and lower bound is less than $e^{\frac{2^d+1}{2^d(2^d-1)}}$.\\

From Theorem \ref{th:forma} and from Remark \ref{rm:basis}, we have that a group $T_\circ$ with $U(T_\circ)=\Span\{\be_{n+1},\dots,\be_{n+d}\}$ is determined by the maps $\gk_{\be_i}$'s, and in particular by $B_{\be_1},\dots,B_{\be_n}$ ($B_{\be_i}=0$ for $n+1\le i\le n+d$). 
{ Thus}, we need to construct the matrices $B_{\be_1},\dots,B_{\be_n}$ so that:


\begin{itemize}

\item[i)] \label{condizionei}$\dim(U(T_\circ))=d$. Then for any $\bv \in \Span\{\be_1,\dots, \be_n\}$, $B_\bv \neq 0$. Indeed, $\bv \in U(T_\circ)$ if and only if $B_\bv =0$. Moreover, 
		for any $\bv=(v_1,\dots,v_{n+d})\in V$ we have 
		$$
		\gk_\bv = \left[\begin{array}{cc}
		 I_{n} & \sum_{i=1}^n {v_i} B_{\be_i} \\ 0 & I_{d} \end{array}\right]
		 $$

		and so we get that $B_\bv = \sum_{i=1}^n v_i B_{\be_i}$ (see Lemma 5.4 in \cite{calderini2017elementary}).\\
		This implies that every non-null linear combination of $B_{\be_1},\dots,B_{\be_n}$ is non-zero.  Note that, in Theorem \ref{lm:numero}, this condition is given by $\cS_1$.

\item[ii)] $T_\circ$ is abelian. This happens if and only if the  $i$th row of $B_{\be_j}$ is equal to the $j$th row of $B_{\be_i}$, obtaining the set $\cS_2$.

\item[iii)] $T_\circ$ is elementary, that is, $\gt_\bv^2=1_V$ for all $\bv$. Then $\gk_\bv$ fixes $\bv$ and, in particular, $\gk_{\be_i}$ fixes $\be_i$. This is equivalent to having the $i$th row  of  $B_{\be_i}$ equals to $0$, which is expressed by $\cS_3$. 

\end{itemize}

Consider now any matrix $B_{\be_i}$. Its size is $n\times d$, thus any row of $B_{\be_i}$ can be written as an element of $\FF_{2^d}$. Let us define the matrix $\mathfrak{B}_\circ=[B_{\be_1}\dots B_{\be_n}]\in (\FF_{2^d})^{n\times n}$. From Condition i), ii) and iii), above, we have the following properties for  $\mathfrak{B}_\circ$:
\begin{itemize}
\item[(I)] $\mathfrak{B}_\circ$ is full rank over $\FF_2$, { when seen as an element of $(\FF_2)^{n\times nd}$}. This guarantees that $\dim (U(T_\circ)) = d$.
\item[(II)] $\mathfrak{B}_\circ$ is symmetric.  Since the $i$th row of $B_{\be_j}$ is equal to the $j$th row of $B_{\be_i}$ we have that the entry $(i,j)$ of $\mathfrak{B}_\circ$ is equal to the entry $(j,i)$.
\item[(III)] $\mathfrak{B}_\circ$ is zero-diagonal, as the $i$th row of $B_{\be_i}$ is zero for all $i$.
\end{itemize}

\begin{example}
Let $N = 5$ and $d=2$ and consider the operation $\circ$ defined by the following $B_{\be_i}$'s:

\[
B_{\be_1 } = \begin{bmatrix}   0 & 0  \\  1 & 1 \\  1 & 1  \end{bmatrix} , B_{\be_2 } = \begin{bmatrix} 1 & 1 \\  0 & 0 \\ 0 & 1 \end{bmatrix}
B_{\be_3} = \begin{bmatrix}   1 & 1  \\  0 & 1 \\ 0 & 0   \end{bmatrix} ,
 B_{\be_4} = \begin{bmatrix}   0 & 0  \\  0 & 0 \\  0 & 0 \end{bmatrix}  , B_{\be_5} = \begin{bmatrix}   0 & 0  \\  0 & 0 \\  0 & 0 \end{bmatrix}
\]
Since we have $n = N-d = 3$, we just need to focus on $B_{\be_1}$, $B_{\be_2}$ and $B_{\be_3}$.\\[2mm]
Representing the rows of the matrices as elements of $\mathbb{F}_{2^2} = \{0,1,\alpha,\alpha+1\}$, where $\alpha^2 = \alpha + 1$, we can rewrite the matrices as 

\[
B_{\be_1 } = \begin{bmatrix}   0  \\  \alpha + 1  \\ \alpha + 1  \end{bmatrix}, B_{\be_2 } = \begin{bmatrix}  \alpha + 1   \\  0 \\ \alpha \end{bmatrix},
B_{\be_3} = \begin{bmatrix}   \alpha + 1  \\  \alpha \\ 0  \end{bmatrix} 
\]
and so we can rewrite $\mathfrak{B}_{\circ}$ as
\[\mathfrak{B}_{\circ} = \begin{bmatrix} B_{\be_1} & B_{\be_2} & B_{\be_3} \end{bmatrix} = \begin{bmatrix} 0 & \alpha +1 & \alpha +1 \\ \alpha +1  & 0 & \alpha \\ \alpha+1 & \alpha & 0  \end{bmatrix}\]

\end{example}

In Theorem \ref{lm:numero} we have a one-to-one correspondence between the matrices $B_{\be_i}$'s and the points of $\cV(\cI_{n,d})$. Thanks to the $\mathfrak{B}_\circ$ construction, we have a one-to-one correspondence between the matrices $\mathfrak{B}_\circ\in (\FF_{2^d})^{n\times n}$ having the above characteristics (I), (II), (III) and the groups $T_\circ$ with $U(T_\circ)$ generated by
$\be_{n+1},\dots,\be_{n+d}$. 
That is, defining $$\cM_{n,d}=\{\mathfrak{B}_\circ\in (\FF_{2^d})^{n \times n}\mid \mathfrak{B}_\circ \text{ satisfies conditions (I), (II), and (III)}\}$$
we have $|\cM_{n,d}|=|\cV(\cI_{n,d})|$. Thus, we aim to estimate $|\cM_{n,d}|$.\\

We recall the following result given in \cite{macwilliams69ort}.
\begin{proposition}\label{prop:sym}
Let $q$ be a power of $2$. The number of $n \times n$ symmetric invertible matrices over $\FF_q$ with zero diagonal is
$$
\begin{cases} q^{\binom{n}{2}} { \prod_{j=1}^{\lceil{\frac{n-1}{2}}\rceil}} \left( 1 - q^{1 - 2j}\right)  &   \text{if $n$ is even}\\ 
0 &  \text{if $n$ is odd}\\ \end{cases}.
$$
\end{proposition}

\begin{proposition}\label{lowerbound}
Let $N = n + d$ and $V = (\mathbb{F}_2)^N$, with $n \geq 2$ and $d \geq 1$. Let $q = 2^d$. Then we can define a lower bound $\nu (n,d) \leq |\cM_{n,d}|$ by

$$
\nu(n,d) = \begin{cases} q^{\binom{n}{2}} { \prod_{j=1}^{\lceil{\frac{n-1}{2}}\rceil}} \left( 1 - q^{1 - 2j}\right)  &   \text{if $n$ is  even}\\ 
(q^{n-1}-2^{n-1}) q^{\binom{n-1}{2}} { \prod_{j=1}^{\lceil{\frac{n-2}{2}}\rceil}}  \left( 1 - q^{1 - 2j}\right)  &  \text{if $n$ is  odd}\\ \end{cases}
$$
\end{proposition}
\begin{proof}
We require that $\mathfrak{B}_\circ$ is full rank over $\FF_2$, thus if it is invertible this condition is satisfied. Then the value given in Proposition \ref{prop:sym} is a lower bound on the number of all acceptable matrices $\mathfrak{B}_\circ$. However, if $n$ is odd from Proposition \ref{prop:sym} we would have only $0$ as lower bound.

To tackle the $n$-odd case, we want to reduce it to the $n$-even case. We show how it is possible to construct a matrix $\mathfrak{B}_\circ$ for fixed values $n$ and $d$ starting from a given $\mathfrak{B}_\circ'$ defined for values $n'=n-1$, { where $n-1$ is even}, and $d'=d$.
Indeed, let $\mathfrak{B}_\circ'\in (\FF_{2^d})^{n-1 \times n-1}$ be such that it is full rank over $\FF_2$, symmetric and zero-diagonal. We need to construct the first row (the first column is the transpose of this) of $\mathfrak{B}_\circ$, setting all the others equal to the rows of $\mathfrak{B}'_\circ$. 
That is, 
$$
\mathfrak B_\circ=\left[\begin{array}{cc}
0&\bb\\
\bb^\intercal&\mathfrak{B}_\circ'
\end{array}\right].
$$
Then, we need to verify how many possible $\mathfrak{B}_\circ$ we can construct starting from $\mathfrak{B}_\circ'$, or equivalently how many vectors $\bb\in (\FF_{2^d})^{n-1}$ we can use, in order to enforce $\mathfrak{B}_\circ\in\cM_{n,d}$.

As $\mathfrak{B}_\circ'$ is full rank over $\FF_2$, summing the rows of this matrix we can create $2^{n-1}-1$ non-zero vectors. Thus, we can choose $2^{d(n-1)}-2^{n-1}$ different vectors $\bb$ of $(\FF_{2^d})^{n-1}$ to construct $\mathfrak{B}_\circ$. Hence, for any $n$ we have $(2^{d(n-1)}-2^{n-1})|\cM_{n-1,d}|\le|\cM_{n,d}|$ and, in particular, for $n$ odd we have at least $(2^{d(n-1)}-2^{n-1}) \nu(n-1,d)$ possible matrices $\mathfrak{B}_\circ$.
\end{proof}
\begin{proposition}
Let $N = n + d$ and $V = (\mathbb{F}_2)^N$. Then
\[  \lvert \mathcal{M}_{n,d}\rvert  = \begin{cases}
2^d - 1 						&  \text{if  $n = 2$} \\
(2^d + 3)(2^d - 1)(2^d - 2) 				& \text{if $n = 3$}\\
2^{\binom{n}{2}} \prod_{j=1}^{\lceil{\frac{n-1}{2}}\rceil} \left( 1 - 2^{1 - 2j}\right) 		&  \text{if $d = 1$ and $n$ is even}\\
0 		&   \text{if $d = 1$ and $n$ is odd}
\end{cases}\]
\end{proposition}
\begin{proof}
The cases $n=2$ and $n=3$ were already proved in \cite{calderini2017elementary}. We restate here the proof using symmetric matrices, obtaining a shorter and clearer proof.

If $n=2$ then
\[ \mathfrak{B}_\circ = \begin{bmatrix} 0 & \bx \\ \bx & 0 \end{bmatrix}\]
 with $ \bx \in \mathbb{F}_{2^d}$. Of course the only vector that we have to avoid is the zero vector. Thus we can use $2^d-1$ vectors $\bx$ to construct $\mathfrak B_\circ$.

If $n=3$ then
\[ \mathfrak{B}_\circ = \begin{bmatrix} 0 & \by & \bz \\ \by & 0 & \bx \\ \bz & \bx & 0 \end{bmatrix}\]
with $ \bx,\by,\bz \in \mathbb{F}_{2^d}$ such that $\mathfrak{B}_\circ\in \cM_{n,d}$.

{ We need to find all the triples $(\bx,\by,\bz)$ such that the matrix $\mathfrak{B}_\circ$ is full rank over $\FF_2$, that is, summing any rows of it we do not obtain the zero vector.} Let us consider the different cases:
\begin{itemize}
	\item $\bx \neq 0$ : $2^d - 1$ possible values for $\bx$,
	\begin{itemize}
		\item $\by = 0$ : so $\bz \notin \{ 0,\bx\}$ and so we get $2^d - 2$ possible values
		\item $\by = \bx$ : so $\bz \notin \{ 0,\bx\}$ and so we get $2^d - 2$ possible values
		\item $\by \neq \bx$ : so $\bz$ can be any element, $\by \notin \{ 0,\bx \}$. So $2^d(2^d-2)$ possible pairs $(\by,\bz)$.
	\end{itemize}
	\item $\bx = 0$ : we have $\by \neq 0$ and $ z \notin \{ 0,y \}$ and so we get $(2^d-1)(2^d-2)$ possible pairs $(\by,\bz)$.
\end{itemize}

Summing all the possible triple $(\bx,\by,\bz)$, we get
\[ (2^d - 1)(2^d - 2 + 2^d - 2 + 2^d (2^d -2)) + (2^d -1)(2^d - 2) = (2^d-1)(2^d-2)(2^d + 3) \]

Now, let $n\ge 2$ and fix $d=1$. As we consider $\mathbb{F}_{2^d} = \mathbb{F}_2$, we have that $\cM_{n,d}$ is the set of symmetric matrices having all zeros on the diagonal and invertible. From Proposition \ref{prop:sym} we have our claim.
\end{proof}

Let us compare the upper bound and the lower bound on the cardinality of $\cM_{n,d}$. 
We recall a theorem on the infinite product convergence criteria, corollary of the monotone convergence theorem (more details can be found in \cite{stewart2015calculus}).
\begin{lemma}\-\\\label{limite}
Let $\{a_j\}_{j\in \mathbb{N}} \subseteq \mathbb{R}_+$. Then
$\displaystyle\prod_{j=1}^{\infty} a_j$ converges if and only if $\displaystyle\sum_{j=1}^{\infty} \ln (a_j)$ converges.\\
Moreover, assume $a_j \geq 1$, and denote  $a_j = 1 + p_j$. Then
\[ 1 + \sum_{j=1}^\infty p_j\leq {\displaystyle\prod_{j=1}^\infty} (1 + p_j) \leq e^{\sum_{j=1}^\infty p_j}, \]
that is, the infinite product converges if and only if the infinite sum of the $p_j$ converges.
\end{lemma}

\begin{proposition}\label{prop:ratio}
Let $\mu(n,d)$ be the bound given in Proposition \ref{prop:upperbound} and let $\nu(n,d)$ be the bound given in Proposition \ref{lowerbound}. Let $q = 2^d$ and $d\ge 2$. If $n$ is even, then
\[\frac{\mu(n,d)}{\nu(n,d)} \le {\displaystyle\prod_{j=1}^\infty} \frac{1}{(1-q^{1-2j})}  \leq e^{\frac{q+1}{q(q-1)}}\]
 if $n$ is odd, then 
\[\frac{\mu(n,d)}{\nu(n,d)} \le {2\displaystyle\prod_{j=1}^\infty} \frac{1}{(1-q^{1-2j})}  \leq 2 e^{\frac{q+1}{q(q-1)}}.\]
\end{proposition}
\begin{proof}
We have
$$
\begin{aligned}
\mu(n,d)=&q^{\binom{n}{2}}-1-\sum_{r=1}^{n-2}{n\choose r}\left(q-1\right)^{\binom{n-r}{2}}\\
	    \le&q^{\binom{n}{2}}.
\end{aligned}
$$
Consider the case $n$ even
$$
\begin{aligned}
\frac{\mu(n,d)}{\nu(n,d)}& \le \frac{q^{\binom{n}{2}}}{q^{\binom{n}{2}} { \prod_{j=1}^{\lceil{\frac{n-1}{2}}\rceil}} \left( 1 - q^{1 - 2j}\right)}\\
&= \prod_{j=1}^{\lceil{\frac{n-1}{2}}\rceil}\frac{1}{{ \left( 1 - q^{1 - 2j}\right)}}\\
&\le\prod_{j=1}^{\infty}\frac{1}{{ \left( 1 - q^{1 - 2j}\right)}},
\end{aligned}
$$
since we can write
\[ \frac{1}{1-q^{1-2j}} = \frac{q^{2j-1}}{q^{2j-1}-1} = 1 + \frac{q}{q^{2j} - q}. \]
Defining 
\[ p_j := \frac{q}{q^{2j} - q}\]
we trivially have that $p_n \geq 0$ and

\[ 1 + \sum_{j=1}^\infty p_j \leq {\displaystyle\prod_{j=1}^\infty} \frac{1}{(1-q^{1-2j})} \leq e^{\sum_{j=1}^\infty p_j}. \]

We have for any $j \ge 2$
\[ p_j \leq \frac{1}{q^j}.\]
From this, we get
\[\sum_{i=1}^\infty p_j \leq p_1+\sum_{i=2}^\infty \frac{1}{q^j}= \frac{2}{q-1}-\frac{1}{q}=\frac{q+1}{q(q-1)} \]
and so
\[ {\displaystyle\prod_{j=1}^\infty} \frac{1}{(1-q^{1-2j})} \leq e^{\frac{q+1}{q(q-1)} }. \]

Now let $n$ be odd,
$$
\begin{aligned}
\frac{\mu(n,d)}{\nu(n,d)}& \le \frac{q^{\binom{n}{2}}}{(q^{n-1}-2^{n-1})\cdot q^{\binom{n-1}{2}} { \prod_{j=1}^{\lceil{\frac{n-2}{2}}\rceil}}  \left( 1 - q^{1 - 2j}\right)}.
\end{aligned}
$$
Since $\binom{j}{i}=\binom{j-1}{i-1}+\binom{j-1}{i}$, $n\ge 2$ and $d\ge 2$  we have
$$
\begin{aligned}
\frac{\mu(n,d)}{\nu(n,d)}& \le \frac{q^{{n-1}}}{(q^{n-1}-2^{n-1})\cdot { \prod_{j=1}^{\lceil{\frac{n-2}{2}}\rceil}}  \left( 1 - q^{1 - 2j}\right)}\\
& = \frac{1}{\left(1-\frac{2^{n-1}}{q^{n-1}}\right)\cdot { \prod_{j=1}^{\lceil{\frac{n-2}{2}}\rceil}}  \left( 1 - q^{1 - 2j}\right)}\\
&\le\frac{1}{\left(1-\frac{2^{n-1}}{q^{n-1}}\right)}\prod_{j=1}^{\infty}\frac{1}{{ \left( 1 - q^{1 - 2j}\right)}}\\
&\le2\prod_{j=1}^{\infty}\frac{1}{{ \left( 1 - q^{1 - 2j}\right)}}\le 2 e^{\frac{q+1}{q(q-1)}}.
\end{aligned}
$$
\end{proof}

Note that in Proposition \ref{prop:ratio}, the comparison for the case $d=1$ is avoided as the lower bound is the exact value of $|\cM_{n,d}|$.

\begin{corollary}
For all possible values of $n\ge 2$ and $d\ge 2$ we have
$$
\lim_{n\rightarrow \infty}\frac{\mu(n,d)}{\nu(n,d)}=\prod_{j=1}^{\infty}\frac{1}{{ \left( 1 - q^{1 - 2j}\right)}}
$$
and
$$
\lim_{d\rightarrow \infty}\frac{\mu(n,d)}{\nu(n,d)}=1.
$$
\end{corollary}
\begin{proof}
Consider the case $n$ odd, then 
$$
\begin{aligned}
\frac{\mu(n,d)}{\nu(n,d)}& = \frac{q^{\binom{n}{2}}\left(1-\frac{1}{q^{\binom{n}{2}}}-\frac{\sum_{r=1}^{n-2}{n\choose r}\prod_{i=1}^{\binom{n-r}{2}}\left(q-1\right)}{q^{\binom{n}{2}}}\right)}{(q^{n-1}-2^{n-1})\cdot q^{\binom{n-1}{2}} { \prod_{j=1}^{\lceil{\frac{n-2}{2}}\rceil}}  \left( 1 - q^{1 - 2j}\right)}\\
& = \frac{\left(1-\frac{1}{q^{\binom{n}{2}}}-\frac{\sum_{r=1}^{n-2}{n\choose r}\prod_{i=1}^{\binom{n-r}{2}}\left(q-1\right)}{q^{\binom{n}{2}}}\right)}{\left(1-\frac{2^{n-1}}{q^{n-1}}\right)\cdot { \prod_{j=1}^{\lceil{\frac{n-2}{2}}\rceil}}  \left( 1 - q^{1 - 2j}\right)}\\
&=\frac{\left(1-\frac{1}{q^{\binom{n}{2}}}-\frac{\sum_{r=1}^{n-2}{n\choose r}\prod_{i=1}^{\binom{n-r}{2}}\left(q-1\right)}{q^{\binom{n}{2}}}\right)}{\left(1-\frac{2^{n-1}}{q^{n-1}}\right)}\prod_{j=1}^{\lceil{\frac{n-2}{2}}\rceil}\frac{1}{{ \left( 1 - q^{1 - 2j}\right)}}
\end{aligned}
$$
Consider 
\begin{equation}\label{eq:limit}
\begin{aligned}
\frac{\sum_{r=1}^{n-2}{n\choose r}\prod_{i=1}^{\binom{n-r}{2}}\left(q-1\right)}{q^{\binom{n}{2}}}&=\frac{\sum_{r=1}^{n-2}{n\choose r}\left(q-1\right)^{\binom{n-r}{2}}}{q^{\binom{n}{2}}}\\
&\le\sum_{r=1}^{n-2}{n\choose r}q^{\binom{n-1}{2}-\binom{n}{2}}\\
&\le \frac{2^{n-2}}{q^{n-1}}.
\end{aligned}
\end{equation}
This implies that the limit of \eqref{eq:limit}, as $n$ approaches infinity, is $0$ and then the limit
$$
\lim_{n\rightarrow \infty}\frac{q^{\binom{n}{2}}\left(1-\frac{1}{q^{\binom{n}{2}}}-\frac{\sum_{r=1}^{n-2}{n\choose r}\prod_{i=1}^{\binom{n-r}{2}}\left(q-1\right)}{q^{\binom{n}{2}}}\right)}{(q^{n-1}-2^{n-1})\cdot q^{\binom{n-1}{2}} { \prod_{j=1}^{\lceil{\frac{n-2}{2}}\rceil}}  \left( 1 - q^{1 - 2j}\right)}=\prod_{j=1}^{\infty}\frac{1}{{ \left( 1 - q^{1 - 2j}\right)}}.
$$

The case $n$ even is similar.

Moreover, from the proof of Proposition \ref{prop:ratio} we have
$$
1\le\frac{\mu(n,d)}{\nu(n,d)}\le
\begin{cases}
 e^{\frac{q+1}{q(q-1)}} &\text{ if $n$ is even}\\
  \frac{1}{\left(1-\frac{2^{n-1}}{q^{n-1}}\right)}e^{\frac{q+1}{q(q-1)}}&\text{ if $n$ is odd}
\end{cases}
$$
and immediately we have that 
$$
\lim_{d\rightarrow \infty}\frac{\mu(n,d)}{\nu(n,d)}=1.
$$
\end{proof}
{ \begin{remark}
From the results obtained in this section we can see that if we would tackle Problem 3. given in Section 2, then it is not possible to search among all the possible practical hidden sums of a given space $V$. Indeed, it will be computationally hard, given the huge amount of these operations, also in small dimensions.
\end{remark}
}
\section{On hidden sums for linear maps}
In this section we investigate Problem 3 given in Section 2 page \pageref{prob3}. In particular we want to see if, for a given $\gl\in\GL(V,+)$, it is possible to individuate an alternative sum $\circ $ such that $\gl \in \GL(V,\circ)$.
\begin{proposition}\label{prop:mix1}
Let $T_\circ\subseteq\AGL(V,+)$ and $\gl\in \GL(V,+)\cap\GL(V,\circ)$ then $U(T_\circ)$ is invariant under the action of $\gl$, i.e. $U(T_\circ)\gl=U(T_\circ)$.
\end{proposition}
\begin{proof}
Recall that $\gl$ is linear with respect to both $+$ and $\circ$. Let $\by\in U(T_\circ)$, { and so $\x\circ\by=\bx+\by$ for any $\bx$}. Thus, for all $\bx$ we have
$$
\bx \circ\by\gl=\bx\gl^{-1}\gl\circ \by\gl=(\bx\gl^{-1}\circ \by)\gl=(\bx\gl^{-1}+ \by)\gl=\bx+ \by\gl.
$$
That implies $\by\gl\in U(T_\circ)$, and so $U(T_\circ)\gl\subseteq U(T_\circ)$.
\end{proof}

\begin{proposition}\label{prop:mix2}
Let $T_\circ\subseteq\AGL(V,+)$ and $\gl\in\GL(V,+)$. Then $\gl$ is in $\GL(V,+)\cap\GL(V,\circ)$ if and only if for all $\bx\in V$ we have
\begin{equation}\label{eq:comm}
\gk_\bx\gl=\gl\gk_{\bx\gl},
\end{equation}
where $\gt_\bx=\gk_\bx\gs_\bx$.
\end{proposition}
\begin{proof}
Let $\by$ be fixed. Then for all $\bx$ we have
\[ (\bx \circ \by)\lambda = \bx\lambda \circ \by\lambda \]
	where
	$$
	\begin{aligned}
	 	(\bx \circ \by)\lambda &= (\bx \gk_\by + \by)\lambda\\
	 	&= \bx\gk_\by\lambda + \by\lambda
	\end{aligned} 
	$$
	and 
	$$
	\begin{aligned}
	 	\bx\lambda \circ \by\lambda &= \bx\lambda \gk_{\by\lambda} + \by\lambda .
	\end{aligned} 
	$$
	Imposing the equality we get
\[\bx ( \gk_\by \lambda) = \bx \lambda \gk_{\by\lambda}.\]
Vice versa let $\bx,\by\in V$ 
$$
\begin{aligned}
(\bx\circ\by)\gl&=\bx\gk_\by\gl+\by\gl\\
&=\bx\gl\gk_{\by\gl}+\by\gl\\
&=\bx\gl\circ\by\gl.
\end{aligned}
$$
\end{proof}

Now we will characterize the linear maps which are also linear for an operation $\circ$, such that $U(T_\circ)$ is generated by the last elements of the canonical basis.
\begin{proposition}\label{prop:lambda}
Let $V=(\FF_2)^N$, with $N=n+d$, $n\ge 2$ and $d\ge 1$. Let $T_\circ\subseteq\AGL(V,+)$ with $U(T_\circ)=\Span\{\be_{n+1},...,\be_{n+d}\}$. Let $\gl\in\GL(V,+)$. Then $\gl\in\GL(V,+)\cap\GL(V,\circ)$ if and only if
$$
\gl=\left[\begin{array}{cc}
\Lambda_1&\Lambda_2\\
0&\Lambda_3
\end{array}\right],
$$
with $\Lambda_1\in \GL((\FF_2)^n)$, $\Lambda_3\in \GL((\FF_2)^d)$, $\Lambda_2$ any matrix { in $ (\FF_2)^{n\times d}$} and for all $\bx \in V$ $B_\bx \Lambda_3=\Lambda_1B_{\bx\gl}$ (see Theorem \ref{th:forma} for the notation of $B_\bx$). 
\end{proposition}
\begin{proof}
Let 
$$
\gl=\left[\begin{array}{cc}
\Lambda_1&\Lambda_2\\
\Lambda_4&\Lambda_3
\end{array}\right],
$$
$\Lambda_1\in (\FF_2)^{n\times n}$, $\Lambda_3\in (\FF_2)^{d\times d}$, $\Lambda_2\in (\FF_2)^{n\times d}$ and $\Lambda_4\in (\FF_2)^{d\times n}$

From Proposition \ref{prop:mix1} we have that $\Lambda_4=0$ as $U(T_\circ)$ is $\gl$-invariant. Thus, as $\gl$ is invertible $\Lambda_1\in \GL((\FF_2)^n)$ and $\Lambda_3\in \GL((\FF_2)^d)$.

By standard matrix multiplication it can be verified that \eqref{eq:comm} and $B_\bx \Lambda_3=\Lambda_1B_{\bx\gl}$ are equivalent. Then, we need to show that this condition does not depend on the matrix $\Lambda_2$. Indeed, let $\bx=(x_1,...,x_n,x_{n+1},...,x_{n+d})\in V$, and define $\bar\bx=(x_1,...,x_n)$ and $\bx'=(x_{n+1},...,x_{n+d})$, i.e. $\bx=(\bar\bx,\bx')$. As reported in Section 3 page  \pageref{condizionei}, we have $\gk_\bx=\gk_{(\bar\bx,0)}$, which implies $B_\bx=B_{(\bar\bx,0)}$, for all $\bx\in V$.
Then, $B_\bx \Lambda_3=\Lambda_1B_{\bx\gl}$ for all $\bx\in V$ if and only if $B_{(\bar\bx,0)} \Lambda_3=\Lambda_1B_{(\bar\bx\Lambda_1,0)}$ for all $\bar\bx=(x_1,...,x_n)$.
So we have 
$$
\left[\begin{array}{cc}
\Lambda_1&\Lambda_2\\
0&\Lambda_3
\end{array}\right]\in\GL(V,\circ),
$$
for any $\Lambda_2$.
\end{proof}

In the following, as $\Lambda_2$ in Proposition \ref{prop:lambda} could be any matrix, we will consider linear maps of type
$$
\left[\begin{array}{cc}
\Lambda_1&*\\
0&\Lambda_3
\end{array}\right]\in\GL(V,\circ),
$$
with $\Lambda_1\in \GL((\FF_2)^n)$, $\Lambda_3\in \GL((\FF_2)^d)$ and $*$ denotes for any matrix of size $n\times d$.

\begin{remark}\label{rm:bas}
From the propositions above, if we want to find an operation $\circ$ that linearizes a linear map $\gl\in\GL(V,+)$, i.e. we want to enforce
$$
\left[\begin{array}{cc}
\Lambda_1&*\\
0&\Lambda_3
\end{array}\right]\in\GL(V,\circ),
$$
then we have to construct some matrices $B_\bx$'s such that $B_\bx \Lambda_3=\Lambda_1B_{\bx\gl}$ for all $\bx$. Moreover, as the standard vectors $\be_i$'s form a basis for the operation $\circ$, then we need to individuate only the matrices $B_{\be_i}$, so that $B_{\be_i} \Lambda_{ 3}=\Lambda_1B_{\be_i\gl}$, and in particular that
$$
B_{\be_i} \Lambda_{ 3}=\Lambda_1B_{\be_i\gl}=\Lambda_1\left(\sum_{i=1}^{n}c_iB_{\be_i}\right),
$$
where $c_1,\dots,c_n$ are the first components of the vector $\be_i\gl$.
\end{remark}
%
%
~\\

\begin{algorithm} \label{searchop}
~ \\
INPUT: 
$$\gl=\left[\begin{array}{cc}
\Lambda_1&*\\
0&\Lambda_3
\end{array}\right],
$$
with $\Lambda_1\in \GL((\FF_2)^n)$, $\Lambda_{ 3}\in \GL((\FF_2)^d)$\\
OUTPUT: all {practical} hidden sums (with $\gk_{\be_i}$ as in \eqref{eq:gammai}) such that:
\begin{itemize}
\item $T_\circ\subseteq\AGL(V,+)$, $T_+\subseteq\AGL(V,\circ)$,
\item $U(T_\circ)$ contains $\be_{n+1},...,\be_{n+d}$,
\item $\gl\in\GL(V,\circ)$.
\end{itemize}
ALGORITHM STEPS:\\
\begin{itemize}
\item[I)] Consider the canonical basis $\be_{1},...,\be_{n+d}$ and compute $\be_{1}\gl,...,\be_{n}\gl$
\item[II)] Solve the linear system given by the equations:
\begin{enumerate}
\item for all $i=1,...,n$
$$
B_{\be_i} \Lambda_{ 3}=\Lambda_1B_{\be_i\gl}=\Lambda_1\left(\sum_{i=1}^{n}c_iB_{\be_i}\right),
$$
(where $c_1,...,c_n$ are the first components of the vector $\be_i\gl$).

\item for all $i=n+1,...,n+d$
$$
B_{\be_i} =0,
$$

\item for all $i=1,...,n$
$$
\bar\be_iB_{\be_i} =0,
$$
(here $\bar\be_i$ is the truncation of $\be_i$ with respect to the first $n$ coordinates)
\item  for all $i,j=1,...,n$
$$
\bar\be_iB_{\be_j} =\bar\be_jB_{\be_i},
$$
\end{enumerate}
\item[III)] return the solutions $\{B_{\be_i}\}_{i=1,...,n+d}$.
\end{itemize}
\end{algorithm}
\noindent{{\em Correctness of Algorithm \ref{searchop}:}\\
Note that thanks to the form of the maps $\gk_{\be_i}$ given in Theorem \ref{th:forma}, we always have $\gk_{\be_i}^2=1_V$, whatever the matrix $B_{\be_i}$ is. This property, with condition $3$ in the algorithm, implies $\gt_{\be_i}^2=1_V$, so $T_\circ=\langle\gt_{\be_{1}},...,\gt_{\be_{n+d}} \rangle$ is elementary.\\
 Condition $4$ guarantees that $T_\circ$ is abelian. Then, $T_\circ$ is also regular (see Corollary 3.8 in \cite{calderini2017elementary} for more details).
This implies that $T_\circ=\langle\gt_{\be_{1}},...,\gt_{\be_{n+d}} \rangle$ is a practical hidden sum.\\
The first condition, as seen in Proposition \ref{prop:lambda}, is equivalent to having $\gk_{\be_i}\gl=\gl\gk_{{\be_i}\gl}$ for all $\be_i$, and for Remark \ref{rm:bas}, we have also $\gk_{\bx}\gl=\gl\gk_{{\bx}\gl}$ for any $\bx$. Then, from Proposition \ref{prop:mix2} we have $\gl\in\GL(V,\circ)$. To conclude, condition $2$ implies that $U(T_\circ)$ contains $\be_{n+1},...,\be_{n+d}$.

Viceversa, consider $T_\circ$ a {practical} hidden sums (with $\gk_{\be_i}$ as in \eqref{eq:gammai}) such that:
\begin{itemize}
\item $T_\circ\subseteq\AGL(V,+)$, $T_+\subseteq\AGL(V,\circ)$,
\item $U(T_\circ)$ contains $\be_{n+1},...,\be_{n+d}$,
\item $\gl\in\GL(V,\circ)$.
\end{itemize}
 We need to check that $T_\circ$ is an output of the algorithm. Equivalently, we need to check that the matrices $B_{\be_i}$ associated to this group satisfies the condition of the system in Algorithm \ref{searchop}.

Since $T_\circ$ is elementary and abelian, conditions $3$ and $4$ are satisfied. Also, condition $2$ holds because $U(T_\circ)$ contains $\be_{n+1},...,\be_{n+d}$. For the first condition, since $\gl\in\GL(V,\circ)$ we have $\gk_{\bx}\gl=\gl\gk_{{\bx}\gl}$ for any $\bx$, which is equivalent to $B_\bx \Lambda_3=\Lambda_1B_{\bx\gl}$ for all $\bx$, and in particular for all $\be_i$.
}\\

Note that, from Algorithm \ref{searchop} we obtain operations $\circ$ such that \\$\{\be_{n+1},\dots\be_{n+d}\}\subseteq U(T_\circ)$. Indeed, we required that for $n+1\le i\le n+d$, $B_{\be_i} =0$, but we did not require that the combinations of $B_{\be_{1}},...,B_{\be_{n}}$ are non-zero. So, if we want to construct hidden sums with $\dim(U(T_\circ))=d$, we just process the solution of Algorithm \ref{searchop} and discharge the $\mathfrak{B}_\circ$'s such that $\mathrm{rank}_{\FF_2}(\mathfrak{B}_\circ)<n$.
%

\subsection{Complexity of our search algorithm}\label{rm:vect}
To solve the linear system in our algorithm, we represent the matrices as vectors:

{\footnotesize
\[ \begin{bmatrix}
b_{1,1} & \cdots & b_{1,d}\\
\vdots & \ddots & \vdots \\
b_{n,1} & \cdots & b_{n,d}
\end{bmatrix} \longleftrightarrow \begin{bmatrix}
\bovermat{1\text{-st row}}{b_{1,1} , ... ,b_{1,d}} , \bovermat{2\text{-nd row}}{b_{2,1} ,... , b_{2,d}} ,& ...& , \bovermat{n\text{-th row}}{b_{n,1} , ... , b_{n,d}}
\end{bmatrix}\]
}

\begin{itemize}
	\item From $i=1,...,n$, $\sum_{j} c_j\Lambda_1B_{\be_j} =  B_{\be_i} \Lambda_3$ we have $n^2  d$ linear equation in $n^2  d$ variables.
	\item From $\bar\be_i B_{\be_i} = 0$ and $\bar\be_i B_{\be_j} = \bar\be_j B_{\be_i}$ we obtain $\binom{n+1}{2}d$ linear equations, in the same variables.
\end{itemize}
So, we only need to find a solution of a binary linear system  of size $$\left(n^2d + \binom{n+1}{2}d\right) \cdot n^2d.$$\\

Therefore, we have immediately the following result

\begin{proposition}
The time complexity of Algorithm \ref{searchop} is $\mathcal{O}\left( n^6 d^3 \right)$ and the space complexity is $\mathcal{O}\left( l \cdot 2^{d-1} n^2  \right)$ where $l$ is the dimension of the solution subspace.
\end{proposition}
%
%
%
In Table \ref{compsearchn} we report some timings for different dimensions of the message space $V$, fixing the value of $d$ equal to $2$. 

\begin{table}[!h]
\centering
\begin{tabular}{|c|c|c|}
\hline
Dimension of $V$ & $n$, $d$ & Timing in second\\
\hline\hline
$64$ & $62$, $2$ &  32.620 seconds \\
$80$ & $78$, $2$ &   84.380 seconds \\
$96$ & $94$, $2$ &    188.200 seconds \\
$112$ & $110$, $2$ &  338.590 seconds \\
$128$ & $126$, $2$ &  616.670 seconds\\
\hline
\end{tabular}
\caption{Computation timing for a Mac Book Pro 15'' early 2011, 4 GB Ram, Intel i7 2.00~Ghz.}\label{compsearchn}
\end{table}
%

\section{Hidden sums for PRESENT's mixing layer}

%
%
%

Here we report our results on the search for a hidden sum suitable for the mixing layer of PRESENT, $\gl_P$, which is  defined by the permutation reported in Table \ref{tab:gl}.

\begin{table}[h]
\centering
\begin{tabular}{|c|c|c|c|c|c|c|c|c|c|c|c|c|c|c|c|c|}
\hline
$i$&1&2&3&4&5&6&7&8&9&10&11&12&13&14&15&16\\

$i\lambda_P$&1&17&33&49&2&18&34&50&3&19&35&51&4&20&36&52\\
\hline
\hline
$i$&17&18&19&20&21&22&23&24&25&26&27&28&29&30&31&32\\

$i\lambda_P$&5&21&37&53&6&22&38&54&7&23&39&55&8&24&40&56\\
\hline
\hline
$i$&33&34&35&36&37&38&39&40&41&42&43&44&45&46&47&48\\

$i\lambda_P$&9&25&41&57&10&26&42&58&11&27&43&59&12&28&44&60\\
\hline
\hline
$i$&49&50&51&52&53&54&55&56&57&58&59&60&61&62&63&64\\

$i\lambda_p$&13&29&45&61&14&30&46&62&15&31&47&63&16&32&48&64\\
\hline
\end{tabular}\caption{PRESENT mixing layer}\label{tab:gl}
\end{table}

Algorithm \ref{searchop} requires in input a linear function in the block form:
\begin{equation}\label{eq:la}
 \lambda' = \left(\begin{matrix} \Lambda_1 & * \\ 0 & \Lambda_3 \end{matrix}\right) 
 \end{equation}
with $\Lambda_1\in\GL((\FF_2)^n)$ and $\Lambda_3\in\GL((\FF_2)^d)$ for some integers $n$ and $d$.

Note that a matrix as in \eqref{eq:la} is such that the space $U'=\Span\{\be_{n+1},\dots,\be_{n+d}\}$ is $\lambda'$-invariant. So, in order to transform a mixing layer $\gl$ into one as in \eqref{eq:la}, it is necessary to individuate a subspace $U$ such that $U\gl=U$. Then, by conjugating $\gl$ by a linear map $\pi$ such that $U\pi=\Span\{\be_{n+1},\dots,\be_{n+d}\}$, we will have $\pi \lambda \pi^{-1}$ as in \eqref{eq:la}. From a group $T_\circ$ obtained from Algorithm~\ref{searchop} for the map $\pi \lambda \pi^{-1}$ we will obtain a hidden sum for $\gl$, that is, $ \pi^{-1} T_\circ\pi$.

For this reason we consider the matrix given by the permutation
\[ \pi_P = (1 , {61})({22}, {62})({43},{63}) \in \Sym(\{ i \mid i = 1,...,64 \})\]
so that
\[ \pi_P \lambda_P \pi_P^{-1} = \left(\begin{matrix} \Lambda_1 & 0 \\ 0 & I_{4} \end{matrix}\right) = \hat{\lambda}_P \]
We can now apply Algorithm \ref{searchop} and obtain all the possible operations that linearize $\hat{\lambda}_P$. The operation space will be denoted by $O$.\\
The time required to compute the operation space $O$ is $\sim 10.420$ seconds and it is generated by $2360$ $60$-tuples of $60\times 4$ matrices. So the number of operations that linearize $\hat{\lambda}_P$, in the form described in Theorem \ref{th:forma}, is $| O| = 2^{2360}$.


We take a random operation $\circ$ (Table \ref{rendomopPresent}) obtained by our algorithm. The operation has $\mathrm{rank}_{\FF_2}(\mathfrak{B}_\circ) = 60 = n$ and so the operation is such that $\dim (U(T_\circ)) = 4 = d$. To compress the table, we represent every row of $B_{e_i}$ as a number in $\{0,...,(2^4 -1)\}$, and, as we did in Remark \ref{rm:vect},  row $i$ represents the matrix $B_{\be_i}$.

\begin{table}
\begin{tikzpicture}
\node[rotate=90,text width=0.9\textheight] (a) 
{
{\fontsize{24}{20}\selectfont
\resizebox{\linewidth}{!}{$ \begin{array}{c||cccccccccccccccccccccccccccccccccccccccccccccccccccccccccccc}
 B_{e_{ 1 }} &   {0} & 8 & 0 & 10 & 8 & 9 & 7 & 3 & 0 & 9 & 4 & 15 & 1 & 15 & 7 & 15 & 8 & 14 & 7 & 0 & 4 & 5 & 9 & 0 & 15 & 1 & 10 & 10 & 2 & 15 & 6 & 7 & 0 & 5 & 7 & 5 & 10 & 8 & 6 & 6 & 10 & 4 & 8 & 7 & 1 & 2 & 8 & 0 & 11 & 14 & 9 & 11 & 14 & 7 & 13 & 1 & 5 & 3 & 8 & 0 \\ 
 B_{e_{ 2 }} &   8 & 0 & 7 & 10 & 14 & 7 & 13 & 6 & 7 & 8 & 3 & 12 & 15 & 13 & 11 & 6 & 14 & 15 & 9 & 11 & 14 & 15 & 12 & 6 & 9 & 6 & 0 & 3 & 14 & 8 & 15 & 6 & 10 & 0 & 10 & 4 & 11 & 8 & 3 & 9 & 9 & 0 & 8 & 6 & 7 & 2 & 2 & 8 & 8 & 5 & 14 & 2 & 1 & 6 & 0 & 9 & 12 & 11 & 12 & 8 \\ 
 B_{e_{ 3 }} &   0 & 7 & 0 & 2 & 10 & 4 & 14 & 5 & 15 & 8 & 12 & 15 & 0 & 15 & 15 & 15 & 7 & 7 & 6 & 7 & 3 & 1 & 12 & 6 & 8 & 15 & 12 & 15 & 9 & 11 & 3 & 14 & 15 & 4 & 2 & 9 & 0 & 7 & 6 & 6 & 15 & 2 & 0 & 9 & 14 & 13 & 8 & 0 & 12 & 7 & 6 & 11 & 5 & 9 & 12 & 13 & 15 & 10 & 1 & 0 \\ 
 B_{e_{ 4 }} &   10 & 10 & 2 & 0 & 8 & 8 & 2 & 15 & 12 & 15 & 12 & 5 & 7 & 3 & 12 & 12 & 15 & 12 & 8 & 8 & 3 & 14 & 15 & 14 & 14 & 8 & 11 & 9 & 7 & 4 & 10 & 6 & 0 & 6 & 5 & 0 & 9 & 10 & 0 & 1 & 1 & 2 & 12 & 0 & 8 & 3 & 13 & 1 & 7 & 8 & 6 & 8 & 8 & 8 & 4 & 8 & 0 & 2 & 2 & 11 \\ 
 B_{e_{ 5 }} &   8 & 14 & 10 & 8 & 0 & 15 & 0 & 5 & 7 & 9 & 10 & 14 & 10 & 11 & 4 & 2 & 14 & 14 & 11 & 1 & 7 & 6 & 8 & 6 & 13 & 12 & 3 & 0 & 6 & 6 & 9 & 9 & 7 & 9 & 9 & 12 & 8 & 6 & 0 & 11 & 3 & 0 & 6 & 12 & 12 & 3 & 6 & 8 & 15 & 14 & 7 & 8 & 13 & 8 & 2 & 15 & 11 & 15 & 2 & 8 \\ 
 B_{e_{ 6 }} &   9 & 7 & 4 & 8 & 15 & 0 & 12 & 15 & 7 & 10 & 10 & 6 & 12 & 0 & 0 & 12 & 14 & 0 & 8 & 8 & 0 & 5 & 13 & 9 & 11 & 7 & 5 & 0 & 12 & 14 & 11 & 5 & 3 & 2 & 15 & 1 & 7 & 4 & 10 & 5 & 7 & 11 & 12 & 11 & 9 & 8 & 2 & 4 & 3 & 0 & 0 & 12 & 15 & 1 & 1 & 4 & 11 & 12 & 7 & 14 \\ 
 B_{e_{ 7 }} &   7 & 13 & 14 & 2 & 0 & 12 & 0 & 4 & 4 & 7 & 13 & 6 & 6 & 5 & 10 & 10 & 9 & 11 & 4 & 13 & 2 & 3 & 8 & 0 & 11 & 7 & 2 & 6 & 10 & 14 & 10 & 10 & 8 & 11 & 11 & 9 & 2 & 12 & 11 & 13 & 13 & 8 & 11 & 1 & 8 & 10 & 7 & 15 & 14 & 11 & 11 & 13 & 14 & 6 & 9 & 2 & 13 & 1 & 13 & 5 \\ 
 B_{e_{ 8 }} &   3 & 6 & 5 & 15 & 5 & 15 & 4 & 0 & 7 & 1 & 11 & 3 & 8 & 7 & 6 & 0 & 14 & 12 & 10 & 5 & 0 & 0 & 14 & 6 & 11 & 9 & 14 & 2 & 6 & 0 & 14 & 5 & 9 & 10 & 11 & 0 & 9 & 10 & 14 & 7 & 8 & 8 & 13 & 8 & 6 & 6 & 9 & 2 & 7 & 6 & 15 & 14 & 5 & 11 & 4 & 15 & 4 & 4 & 5 & 14 \\ 
 B_{e_{ 9 }} &   0 & 7 & 15 & 12 & 7 & 7 & 4 & 7 & 0 & 6 & 2 & 6 & 2 & 7 & 9 & 11 & 10 & 3 & 0 & 5 & 4 & 14 & 7 & 9 & 14 & 12 & 6 & 12 & 5 & 6 & 6 & 13 & 15 & 8 & 15 & 15 & 8 & 15 & 2 & 10 & 12 & 12 & 15 & 1 & 15 & 15 & 9 & 0 & 0 & 9 & 14 & 0 & 15 & 11 & 13 & 1 & 15 & 3 & 8 & 0 \\ 
 B_{e_{ 10 }} &   9 & 8 & 8 & 15 & 9 & 10 & 7 & 1 & 6 & 0 & 6 & 3 & 8 & 7 & 9 & 13 & 11 & 7 & 6 & 8 & 8 & 6 & 12 & 7 & 4 & 9 & 1 & 12 & 10 & 2 & 8 & 1 & 0 & 2 & 11 & 1 & 6 & 4 & 0 & 8 & 5 & 15 & 12 & 2 & 5 & 0 & 12 & 10 & 9 & 9 & 4 & 2 & 15 & 8 & 5 & 5 & 10 & 12 & 6 & 7 \\ 
 B_{e_{ 11 }} &   4 & 3 & 12 & 12 & 10 & 10 & 13 & 11 & 2 & 6 & 0 & 12 & 5 & 2 & 12 & 7 & 9 & 7 & 5 & 0 & 15 & 11 & 5 & 10 & 11 & 9 & 14 & 3 & 11 & 12 & 10 & 3 & 15 & 13 & 12 & 13 & 11 & 2 & 11 & 1 & 12 & 13 & 10 & 13 & 3 & 7 & 8 & 10 & 1 & 8 & 5 & 0 & 13 & 9 & 3 & 9 & 9 & 14 & 0 & 7 \\ 
 B_{e_{ 12 }} &   15 & 12 & 15 & 5 & 14 & 6 & 6 & 3 & 6 & 3 & 12 & 0 & 6 & 10 & 15 & 0 & 7 & 9 & 5 & 13 & 0 & 11 & 9 & 15 & 11 & 7 & 8 & 3 & 15 & 0 & 3 & 9 & 14 & 8 & 3 & 12 & 4 & 6 & 7 & 1 & 5 & 0 & 9 & 9 & 7 & 5 & 0 & 1 & 8 & 6 & 7 & 13 & 13 & 1 & 12 & 12 & 1 & 6 & 11 & 9 \\ 
 B_{e_{ 13 }} &   1 & 15 & 0 & 7 & 10 & 12 & 6 & 8 & 2 & 8 & 5 & 6 & 0 & 8 & 0 & 8 & 8 & 3 & 9 & 8 & 8 & 6 & 10 & 8 & 2 & 15 & 0 & 4 & 15 & 14 & 1 & 8 & 12 & 14 & 1 & 0 & 15 & 8 & 2 & 2 & 12 & 11 & 12 & 2 & 5 & 9 & 0 & 11 & 7 & 7 & 8 & 12 & 3 & 4 & 3 & 14 & 12 & 10 & 13 & 10 \\ 
 B_{e_{ 14 }} &   15 & 13 & 15 & 3 & 11 & 0 & 5 & 7 & 7 & 7 & 2 & 10 & 8 & 0 & 2 & 4 & 1 & 15 & 15 & 4 & 8 & 14 & 10 & 3 & 13 & 4 & 0 & 0 & 5 & 5 & 11 & 0 & 5 & 14 & 13 & 11 & 8 & 12 & 11 & 13 & 0 & 8 & 0 & 0 & 13 & 5 & 14 & 14 & 8 & 5 & 13 & 3 & 4 & 1 & 8 & 5 & 5 & 6 & 13 & 0 \\ 
 B_{e_{ 15 }} &   7 & 11 & 15 & 12 & 4 & 0 & 10 & 6 & 9 & 9 & 12 & 15 & 0 & 2 & 0 & 3 & 12 & 11 & 10 & 5 & 1 & 3 & 3 & 14 & 9 & 0 & 6 & 10 & 0 & 15 & 1 & 0 & 15 & 13 & 9 & 9 & 1 & 12 & 14 & 2 & 13 & 0 & 5 & 4 & 12 & 0 & 1 & 5 & 0 & 4 & 1 & 3 & 11 & 2 & 8 & 8 & 9 & 10 & 4 & 5 \\ 
 B_{e_{ 16 }} &   15 & 6 & 15 & 12 & 2 & 12 & 10 & 0 & 11 & 13 & 7 & 0 & 8 & 4 & 3 & 0 & 8 & 12 & 12 & 0 & 12 & 4 & 4 & 4 & 13 & 4 & 13 & 0 & 14 & 5 & 11 & 11 & 0 & 11 & 10 & 5 & 2 & 0 & 5 & 10 & 0 & 2 & 1 & 4 & 13 & 8 & 8 & 8 & 12 & 13 & 9 & 1 & 3 & 8 & 15 & 6 & 3 & 7 & 12 & 11 \\ 
 B_{e_{ 17 }} &   8 & 14 & 7 & 15 & 14 & 14 & 9 & 14 & 10 & 11 & 9 & 7 & 8 & 1 & 12 & 8 & 0 & 7 & 8 & 13 & 15 & 9 & 6 & 8 & 0 & 8 & 0 & 2 & 5 & 6 & 11 & 15 & 7 & 13 & 3 & 11 & 9 & 12 & 0 & 15 & 10 & 3 & 2 & 2 & 14 & 0 & 12 & 8 & 10 & 6 & 12 & 6 & 11 & 6 & 3 & 6 & 4 & 9 & 6 & 8 \\ 
 B_{e_{ 18 }} &   14 & 15 & 7 & 12 & 14 & 0 & 11 & 12 & 3 & 7 & 7 & 9 & 3 & 15 & 11 & 12 & 7 & 0 & 10 & 0 & 0 & 4 & 7 & 14 & 2 & 4 & 11 & 8 & 0 & 1 & 12 & 5 & 4 & 12 & 10 & 0 & 8 & 13 & 5 & 11 & 15 & 10 & 12 & 2 & 0 & 1 & 7 & 9 & 8 & 15 & 6 & 12 & 8 & 9 & 0 & 5 & 1 & 5 & 11 & 4 \\ 
 B_{e_{ 19 }} &   7 & 9 & 6 & 8 & 11 & 8 & 4 & 10 & 0 & 6 & 5 & 5 & 9 & 15 & 10 & 12 & 8 & 10 & 0 & 7 & 7 & 5 & 9 & 2 & 2 & 4 & 15 & 0 & 9 & 8 & 12 & 6 & 8 & 7 & 6 & 9 & 6 & 12 & 1 & 8 & 11 & 0 & 2 & 12 & 4 & 5 & 6 & 9 & 15 & 1 & 3 & 13 & 8 & 7 & 12 & 1 & 1 & 8 & 2 & 10 \\ 
 B_{e_{ 20 }} &   0 & 11 & 7 & 8 & 1 & 8 & 13 & 5 & 5 & 8 & 0 & 13 & 8 & 4 & 5 & 0 & 13 & 0 & 7 & 0 & 15 & 5 & 4 & 5 & 14 & 12 & 8 & 5 & 5 & 1 & 6 & 14 & 15 & 5 & 2 & 2 & 15 & 10 & 0 & 11 & 13 & 11 & 3 & 14 & 13 & 8 & 13 & 15 & 3 & 7 & 10 & 4 & 4 & 3 & 0 & 0 & 11 & 13 & 0 & 14 \\ 
 B_{e_{ 21 }} &   4 & 14 & 3 & 3 & 7 & 0 & 2 & 0 & 4 & 8 & 15 & 0 & 8 & 8 & 1 & 12 & 15 & 0 & 7 & 15 & 0 & 5 & 4 & 1 & 12 & 13 & 10 & 1 & 15 & 9 & 5 & 4 & 7 & 11 & 7 & 11 & 10 & 7 & 11 & 12 & 10 & 5 & 12 & 7 & 6 & 0 & 11 & 14 & 12 & 12 & 9 & 12 & 0 & 14 & 8 & 5 & 0 & 11 & 2 & 9 \\ 
 B_{e_{ 22 }} &   5 & 15 & 1 & 14 & 6 & 5 & 3 & 0 & 14 & 6 & 11 & 11 & 6 & 14 & 3 & 4 & 9 & 4 & 5 & 5 & 5 & 0 & 10 & 9 & 10 & 13 & 3 & 10 & 5 & 8 & 2 & 1 & 13 & 2 & 9 & 8 & 1 & 1 & 1 & 1 & 3 & 12 & 11 & 7 & 9 & 8 & 11 & 1 & 8 & 15 & 12 & 6 & 0 & 0 & 13 & 1 & 0 & 0 & 0 & 7 \\ 
 B_{e_{ 23 }} &   9 & 12 & 12 & 15 & 8 & 13 & 8 & 14 & 7 & 12 & 5 & 9 & 10 & 10 & 3 & 4 & 6 & 7 & 9 & 4 & 4 & 10 & 0 & 4 & 12 & 12 & 1 & 9 & 10 & 11 & 1 & 13 & 15 & 7 & 9 & 0 & 4 & 12 & 6 & 1 & 2 & 4 & 0 & 2 & 6 & 9 & 7 & 1 & 8 & 9 & 7 & 4 & 12 & 13 & 12 & 1 & 12 & 15 & 0 & 8 \\ 
 B_{e_{ 24 }} &   0 & 6 & 6 & 14 & 6 & 9 & 0 & 6 & 9 & 7 & 10 & 15 & 8 & 3 & 14 & 4 & 8 & 14 & 2 & 5 & 1 & 9 & 4 & 0 & 6 & 13 & 3 & 13 & 11 & 12 & 9 & 8 & 11 & 14 & 12 & 15 & 8 & 11 & 9 & 12 & 9 & 4 & 8 & 1 & 1 & 9 & 9 & 15 & 4 & 0 & 0 & 5 & 1 & 12 & 9 & 0 & 2 & 7 & 2 & 7 \\ 
 B_{e_{ 25 }} &   15 & 9 & 8 & 14 & 13 & 11 & 11 & 11 & 14 & 4 & 11 & 11 & 2 & 13 & 9 & 13 & 0 & 2 & 2 & 14 & 12 & 10 & 12 & 6 & 0 & 8 & 11 & 9 & 4 & 0 & 13 & 2 & 4 & 11 & 13 & 13 & 7 & 7 & 8 & 1 & 13 & 2 & 10 & 13 & 6 & 6 & 1 & 5 & 6 & 10 & 8 & 11 & 5 & 14 & 10 & 3 & 10 & 10 & 7 & 7 \\ 
 B_{e_{ 26 }} &   1 & 6 & 15 & 8 & 12 & 7 & 7 & 9 & 12 & 9 & 9 & 7 & 15 & 4 & 0 & 4 & 8 & 4 & 4 & 12 & 13 & 13 & 12 & 13 & 8 & 0 & 6 & 12 & 14 & 4 & 1 & 1 & 7 & 12 & 2 & 12 & 12 & 12 & 4 & 15 & 5 & 1 & 4 & 0 & 9 & 9 & 2 & 8 & 10 & 10 & 6 & 0 & 10 & 11 & 9 & 10 & 3 & 1 & 7 & 9 \\ 
 B_{e_{ 27 }} &   10 & 0 & 12 & 11 & 3 & 5 & 2 & 14 & 6 & 1 & 14 & 8 & 0 & 0 & 6 & 13 & 0 & 11 & 15 & 8 & 10 & 3 & 1 & 3 & 11 & 6 & 0 & 5 & 14 & 9 & 4 & 12 & 2 & 8 & 13 & 0 & 0 & 4 & 9 & 7 & 11 & 9 & 5 & 0 & 7 & 13 & 13 & 4 & 2 & 8 & 0 & 2 & 11 & 4 & 10 & 1 & 14 & 2 & 5 & 6 \\ 
 B_{e_{ 28 }} &   10 & 3 & 15 & 9 & 0 & 0 & 6 & 2 & 12 & 12 & 3 & 3 & 4 & 0 & 10 & 0 & 2 & 8 & 0 & 5 & 1 & 10 & 9 & 13 & 9 & 12 & 5 & 0 & 4 & 9 & 1 & 8 & 13 & 10 & 7 & 0 & 5 & 9 & 13 & 5 & 3 & 10 & 15 & 4 & 12 & 5 & 3 & 2 & 3 & 6 & 5 & 8 & 8 & 9 & 5 & 13 & 8 & 7 & 8 & 13 \\ 
 B_{e_{ 29 }} &   2 & 14 & 9 & 7 & 6 & 12 & 10 & 6 & 5 & 10 & 11 & 15 & 15 & 5 & 0 & 14 & 5 & 0 & 9 & 5 & 15 & 5 & 10 & 11 & 4 & 14 & 14 & 4 & 0 & 6 & 7 & 15 & 7 & 11 & 8 & 4 & 1 & 9 & 8 & 4 & 11 & 14 & 0 & 5 & 3 & 2 & 8 & 14 & 8 & 6 & 6 & 13 & 7 & 0 & 6 & 0 & 6 & 14 & 9 & 3 \\ 
 B_{e_{ 30 }} &   15 & 8 & 11 & 4 & 6 & 14 & 14 & 0 & 6 & 2 & 12 & 0 & 14 & 5 & 15 & 5 & 6 & 1 & 8 & 1 & 9 & 8 & 11 & 12 & 0 & 4 & 9 & 9 & 6 & 0 & 12 & 0 & 9 & 6 & 9 & 2 & 7 & 13 & 4 & 7 & 10 & 3 & 4 & 2 & 15 & 13 & 1 & 7 & 8 & 11 & 1 & 8 & 3 & 12 & 9 & 9 & 14 & 9 & 9 & 0 \\ 
 B_{e_{ 31 }} &   6 & 15 & 3 & 10 & 9 & 11 & 10 & 14 & 6 & 8 & 10 & 3 & 1 & 11 & 1 & 11 & 11 & 12 & 12 & 6 & 5 & 2 & 1 & 9 & 13 & 1 & 4 & 1 & 7 & 12 & 0 & 0 & 10 & 1 & 14 & 10 & 8 & 15 & 2 & 4 & 1 & 7 & 10 & 4 & 1 & 5 & 4 & 3 & 2 & 4 & 6 & 7 & 13 & 7 & 7 & 1 & 2 & 4 & 0 & 6 \\ 
 B_{e_{ 32 }} &   7 & 6 & 14 & 6 & 9 & 5 & 10 & 5 & 13 & 1 & 3 & 9 & 8 & 0 & 0 & 11 & 15 & 5 & 6 & 14 & 4 & 1 & 13 & 8 & 2 & 1 & 12 & 8 & 15 & 0 & 0 & 0 & 1 & 3 & 11 & 3 & 5 & 10 & 3 & 2 & 9 & 1 & 6 & 11 & 12 & 13 & 0 & 5 & 14 & 0 & 11 & 4 & 5 & 9 & 10 & 1 & 8 & 1 & 7 & 1 \\ 
 B_{e_{ 33 }} &   0 & 10 & 15 & 0 & 7 & 3 & 8 & 9 & 15 & 0 & 15 & 14 & 12 & 5 & 15 & 0 & 7 & 4 & 8 & 15 & 7 & 13 & 15 & 11 & 4 & 7 & 2 & 13 & 7 & 9 & 10 & 1 & 0 & 14 & 12 & 15 & 6 & 12 & 12 & 3 & 2 & 6 & 11 & 8 & 6 & 12 & 1 & 0 & 2 & 5 & 15 & 15 & 7 & 6 & 15 & 14 & 9 & 6 & 9 & 0 \\ 
 B_{e_{ 34 }} &   5 & 0 & 4 & 6 & 9 & 2 & 11 & 10 & 8 & 2 & 13 & 8 & 14 & 14 & 13 & 11 & 13 & 12 & 7 & 5 & 11 & 2 & 7 & 14 & 11 & 12 & 8 & 10 & 11 & 6 & 1 & 3 & 14 & 0 & 13 & 10 & 4 & 8 & 2 & 10 & 11 & 11 & 13 & 7 & 11 & 9 & 13 & 7 & 2 & 4 & 6 & 10 & 13 & 0 & 6 & 10 & 9 & 13 & 1 & 15 \\ 
 B_{e_{ 35 }} &   7 & 10 & 2 & 5 & 9 & 15 & 11 & 11 & 15 & 11 & 12 & 3 & 1 & 13 & 9 & 10 & 3 & 10 & 6 & 2 & 7 & 9 & 9 & 12 & 13 & 2 & 13 & 7 & 8 & 9 & 14 & 11 & 12 & 13 & 0 & 12 & 5 & 5 & 14 & 10 & 12 & 11 & 0 & 8 & 5 & 3 & 0 & 4 & 12 & 11 & 12 & 7 & 0 & 10 & 3 & 3 & 13 & 1 & 13 & 10 \\ 
 B_{e_{ 36 }} &   5 & 4 & 9 & 0 & 12 & 1 & 9 & 0 & 15 & 1 & 13 & 12 & 0 & 11 & 9 & 5 & 11 & 0 & 9 & 2 & 11 & 8 & 0 & 15 & 13 & 12 & 0 & 0 & 4 & 2 & 10 & 3 & 15 & 10 & 12 & 0 & 10 & 3 & 6 & 1 & 9 & 14 & 3 & 1 & 1 & 8 & 4 & 7 & 12 & 6 & 15 & 3 & 5 & 14 & 10 & 0 & 9 & 2 & 4 & 5 \\ 
 B_{e_{ 37 }} &   10 & 11 & 0 & 9 & 8 & 7 & 2 & 9 & 8 & 6 & 11 & 4 & 15 & 8 & 1 & 2 & 9 & 8 & 6 & 15 & 10 & 1 & 4 & 8 & 7 & 12 & 0 & 5 & 1 & 7 & 8 & 5 & 6 & 4 & 5 & 10 & 0 & 9 & 15 & 12 & 6 & 1 & 13 & 6 & 3 & 12 & 2 & 7 & 8 & 10 & 5 & 12 & 7 & 2 & 0 & 6 & 9 & 8 & 12 & 9 \\ 
 B_{e_{ 38 }} &   8 & 8 & 7 & 10 & 6 & 4 & 12 & 10 & 15 & 4 & 2 & 6 & 8 & 12 & 12 & 0 & 12 & 13 & 12 & 10 & 7 & 1 & 12 & 11 & 7 & 12 & 4 & 9 & 9 & 13 & 15 & 10 & 12 & 8 & 5 & 3 & 9 & 0 & 1 & 1 & 9 & 6 & 4 & 7 & 7 & 12 & 0 & 9 & 15 & 14 & 9 & 4 & 4 & 4 & 9 & 13 & 0 & 1 & 2 & 1 \\ 
 B_{e_{ 39 }} &   6 & 3 & 6 & 0 & 0 & 10 & 11 & 14 & 2 & 0 & 11 & 7 & 2 & 11 & 14 & 5 & 0 & 5 & 1 & 0 & 11 & 1 & 6 & 9 & 8 & 4 & 9 & 13 & 8 & 4 & 2 & 3 & 12 & 2 & 14 & 6 & 15 & 1 & 0 & 4 & 13 & 9 & 2 & 13 & 0 & 10 & 5 & 10 & 11 & 14 & 8 & 13 & 8 & 3 & 5 & 12 & 0 & 7 & 0 & 4 \\ 
 B_{e_{ 40 }} &   6 & 9 & 6 & 1 & 11 & 5 & 13 & 7 & 10 & 8 & 1 & 1 & 2 & 13 & 2 & 10 & 15 & 11 & 8 & 11 & 12 & 1 & 1 & 12 & 1 & 15 & 7 & 5 & 4 & 7 & 4 & 2 & 3 & 10 & 10 & 1 & 12 & 1 & 4 & 0 & 14 & 2 & 7 & 4 & 6 & 7 & 0 & 6 & 10 & 14 & 3 & 11 & 6 & 9 & 1 & 0 & 10 & 4 & 4 & 3 \\ 
 B_{e_{ 41 }} &   10 & 9 & 15 & 1 & 3 & 7 & 13 & 8 & 12 & 5 & 12 & 5 & 12 & 0 & 13 & 0 & 10 & 15 & 11 & 13 & 10 & 3 & 2 & 9 & 13 & 5 & 11 & 3 & 11 & 10 & 1 & 9 & 2 & 11 & 12 & 9 & 6 & 9 & 13 & 14 & 0 & 14 & 7 & 0 & 12 & 3 & 13 & 7 & 5 & 11 & 3 & 10 & 2 & 12 & 7 & 11 & 12 & 10 & 8 & 4 \\ 
 B_{e_{ 42 }} &   4 & 0 & 2 & 2 & 0 & 11 & 8 & 8 & 12 & 15 & 13 & 0 & 11 & 8 & 0 & 2 & 3 & 10 & 0 & 11 & 5 & 12 & 4 & 4 & 2 & 1 & 9 & 10 & 14 & 3 & 7 & 1 & 6 & 11 & 11 & 14 & 1 & 6 & 9 & 2 & 14 & 0 & 13 & 5 & 8 & 5 & 0 & 6 & 0 & 14 & 7 & 5 & 0 & 9 & 13 & 3 & 6 & 4 & 13 & 10 \\ 
 B_{e_{ 43 }} &   8 & 8 & 0 & 12 & 6 & 12 & 11 & 13 & 15 & 12 & 10 & 9 & 12 & 0 & 5 & 1 & 2 & 12 & 2 & 3 & 12 & 11 & 0 & 8 & 10 & 4 & 5 & 15 & 0 & 4 & 10 & 6 & 11 & 13 & 0 & 3 & 13 & 4 & 2 & 7 & 7 & 13 & 0 & 12 & 0 & 0 & 4 & 11 & 8 & 14 & 13 & 1 & 4 & 5 & 8 & 4 & 3 & 11 & 8 & 15 \\ 
 B_{e_{ 44 }} &   7 & 6 & 9 & 0 & 12 & 11 & 1 & 8 & 1 & 2 & 13 & 9 & 2 & 0 & 4 & 4 & 2 & 2 & 12 & 14 & 7 & 7 & 2 & 1 & 13 & 0 & 0 & 4 & 5 & 2 & 4 & 11 & 8 & 7 & 8 & 1 & 6 & 7 & 13 & 4 & 0 & 5 & 12 & 0 & 11 & 8 & 8 & 8 & 13 & 9 & 0 & 8 & 13 & 9 & 3 & 0 & 4 & 0 & 8 & 8 \\ 
 B_{e_{ 45 }} &   1 & 7 & 14 & 8 & 12 & 9 & 8 & 6 & 15 & 5 & 3 & 7 & 5 & 13 & 12 & 13 & 14 & 0 & 4 & 13 & 6 & 9 & 6 & 1 & 6 & 9 & 7 & 12 & 3 & 15 & 1 & 12 & 6 & 11 & 5 & 1 & 3 & 7 & 0 & 6 & 12 & 8 & 0 & 11 & 0 & 3 & 9 & 9 & 6 & 15 & 7 & 9 & 10 & 0 & 5 & 11 & 15 & 3 & 0 & 15 \\ 
 B_{e_{ 46 }} &   2 & 2 & 13 & 3 & 3 & 8 & 10 & 6 & 15 & 0 & 7 & 5 & 9 & 5 & 0 & 8 & 0 & 1 & 5 & 8 & 0 & 8 & 9 & 9 & 6 & 9 & 13 & 5 & 2 & 13 & 5 & 13 & 12 & 9 & 3 & 8 & 12 & 12 & 10 & 7 & 3 & 5 & 0 & 8 & 3 & 0 & 4 & 13 & 4 & 4 & 12 & 15 & 0 & 9 & 5 & 10 & 10 & 1 & 3 & 10 \\ 
 B_{e_{ 47 }} &   8 & 2 & 8 & 13 & 6 & 2 & 7 & 9 & 9 & 12 & 8 & 0 & 0 & 14 & 1 & 8 & 12 & 7 & 6 & 13 & 11 & 11 & 7 & 9 & 1 & 2 & 13 & 3 & 8 & 1 & 4 & 0 & 1 & 13 & 0 & 4 & 2 & 0 & 5 & 0 & 13 & 0 & 4 & 8 & 9 & 4 & 0 & 8 & 2 & 5 & 11 & 12 & 0 & 2 & 8 & 7 & 4 & 4 & 8 & 7 \\ 
 B_{e_{ 48 }} &   0 & 8 & 0 & 1 & 8 & 4 & 15 & 2 & 0 & 10 & 10 & 1 & 11 & 14 & 5 & 8 & 8 & 9 & 9 & 15 & 14 & 1 & 1 & 15 & 5 & 8 & 4 & 2 & 14 & 7 & 3 & 5 & 0 & 7 & 4 & 7 & 7 & 9 & 10 & 6 & 7 & 6 & 11 & 8 & 9 & 13 & 8 & 0 & 10 & 3 & 15 & 15 & 0 & 0 & 10 & 7 & 5 & 6 & 7 & 0 \\ 
 B_{e_{ 49 }} &   11 & 8 & 12 & 7 & 15 & 3 & 14 & 7 & 0 & 9 & 1 & 8 & 7 & 8 & 0 & 12 & 10 & 8 & 15 & 3 & 12 & 8 & 8 & 4 & 6 & 10 & 2 & 3 & 8 & 8 & 2 & 14 & 2 & 2 & 12 & 12 & 8 & 15 & 11 & 10 & 5 & 0 & 8 & 13 & 6 & 4 & 2 & 10 & 0 & 15 & 5 & 12 & 8 & 14 & 9 & 6 & 0 & 1 & 0 & 1 \\ 
 B_{e_{ 50 }} &   14 & 5 & 7 & 8 & 14 & 0 & 11 & 6 & 9 & 9 & 8 & 6 & 7 & 5 & 4 & 13 & 6 & 15 & 1 & 7 & 12 & 15 & 9 & 0 & 10 & 10 & 8 & 6 & 6 & 11 & 4 & 0 & 5 & 4 & 11 & 6 & 10 & 14 & 14 & 14 & 11 & 14 & 14 & 9 & 15 & 4 & 5 & 3 & 15 & 0 & 3 & 0 & 5 & 6 & 2 & 5 & 0 & 7 & 8 & 2 \\ 
 B_{e_{ 51 }} &   9 & 14 & 6 & 6 & 7 & 0 & 11 & 15 & 14 & 4 & 5 & 7 & 8 & 13 & 1 & 9 & 12 & 6 & 3 & 10 & 9 & 12 & 7 & 0 & 8 & 6 & 0 & 5 & 6 & 1 & 6 & 11 & 15 & 6 & 12 & 15 & 5 & 9 & 8 & 3 & 3 & 7 & 13 & 0 & 7 & 12 & 11 & 15 & 5 & 3 & 0 & 0 & 13 & 15 & 3 & 9 & 12 & 1 & 9 & 1 \\ 
 B_{e_{ 52 }} &   11 & 2 & 11 & 8 & 8 & 12 & 13 & 14 & 0 & 2 & 0 & 13 & 12 & 3 & 3 & 1 & 6 & 12 & 13 & 4 & 12 & 6 & 4 & 5 & 11 & 0 & 2 & 8 & 13 & 8 & 7 & 4 & 15 & 10 & 7 & 3 & 12 & 4 & 13 & 11 & 10 & 5 & 1 & 8 & 9 & 15 & 12 & 15 & 12 & 0 & 0 & 0 & 0 & 4 & 0 & 11 & 5 & 10 & 4 & 8 \\ 
 B_{e_{ 53 }} &   14 & 1 & 5 & 8 & 13 & 15 & 14 & 5 & 15 & 15 & 13 & 13 & 3 & 4 & 11 & 3 & 11 & 8 & 8 & 4 & 0 & 0 & 12 & 1 & 5 & 10 & 11 & 8 & 7 & 3 & 13 & 5 & 7 & 13 & 0 & 5 & 7 & 4 & 8 & 6 & 2 & 0 & 4 & 13 & 10 & 0 & 0 & 0 & 8 & 5 & 13 & 0 & 0 & 5 & 5 & 14 & 2 & 11 & 14 & 15 \\ 
 B_{e_{ 54 }} &   7 & 6 & 9 & 8 & 8 & 1 & 6 & 11 & 11 & 8 & 9 & 1 & 4 & 1 & 2 & 8 & 6 & 9 & 7 & 3 & 14 & 0 & 13 & 12 & 14 & 11 & 4 & 9 & 0 & 12 & 7 & 9 & 6 & 0 & 10 & 14 & 2 & 4 & 3 & 9 & 12 & 9 & 5 & 9 & 0 & 9 & 2 & 0 & 14 & 6 & 15 & 4 & 5 & 0 & 13 & 8 & 15 & 12 & 1 & 15 \\ 
 B_{e_{ 55 }} &   13 & 0 & 12 & 4 & 2 & 1 & 9 & 4 & 13 & 5 & 3 & 12 & 3 & 8 & 8 & 15 & 3 & 0 & 12 & 0 & 8 & 13 & 12 & 9 & 10 & 9 & 10 & 5 & 6 & 9 & 7 & 10 & 15 & 6 & 3 & 10 & 0 & 9 & 5 & 1 & 7 & 13 & 8 & 3 & 5 & 5 & 8 & 10 & 9 & 2 & 3 & 0 & 5 & 13 & 0 & 8 & 0 & 5 & 4 & 2 \\ 
 B_{e_{ 56 }} &   1 & 9 & 13 & 8 & 15 & 4 & 2 & 15 & 1 & 5 & 9 & 12 & 14 & 5 & 8 & 6 & 6 & 5 & 1 & 0 & 5 & 1 & 1 & 0 & 3 & 10 & 1 & 13 & 0 & 9 & 1 & 1 & 14 & 10 & 3 & 0 & 6 & 13 & 12 & 0 & 11 & 3 & 4 & 0 & 11 & 10 & 7 & 7 & 6 & 5 & 9 & 11 & 14 & 8 & 8 & 0 & 3 & 2 & 11 & 5 \\ 
 B_{e_{ 57 }} &   5 & 12 & 15 & 0 & 11 & 11 & 13 & 4 & 15 & 10 & 9 & 1 & 12 & 5 & 9 & 3 & 4 & 1 & 1 & 11 & 0 & 0 & 12 & 2 & 10 & 3 & 14 & 8 & 6 & 14 & 2 & 8 & 9 & 9 & 13 & 9 & 9 & 0 & 0 & 10 & 12 & 6 & 3 & 4 & 15 & 10 & 4 & 5 & 0 & 0 & 12 & 5 & 2 & 15 & 0 & 3 & 0 & 1 & 1 & 7 \\ 
 B_{e_{ 58 }} &   3 & 11 & 10 & 2 & 15 & 12 & 1 & 4 & 3 & 12 & 14 & 6 & 10 & 6 & 10 & 7 & 9 & 5 & 8 & 13 & 11 & 0 & 15 & 7 & 10 & 1 & 2 & 7 & 14 & 9 & 4 & 1 & 6 & 13 & 1 & 2 & 8 & 1 & 7 & 4 & 10 & 4 & 11 & 0 & 3 & 1 & 4 & 6 & 1 & 7 & 1 & 10 & 11 & 12 & 5 & 2 & 1 & 0 & 4 & 6 \\ 
 B_{e_{ 59 }} &   8 & 12 & 1 & 2 & 2 & 7 & 13 & 5 & 8 & 6 & 0 & 11 & 13 & 13 & 4 & 12 & 6 & 11 & 2 & 0 & 2 & 0 & 0 & 2 & 7 & 7 & 5 & 8 & 9 & 9 & 0 & 7 & 9 & 1 & 13 & 4 & 12 & 2 & 0 & 4 & 8 & 13 & 8 & 8 & 0 & 3 & 8 & 7 & 0 & 8 & 9 & 4 & 14 & 1 & 4 & 11 & 1 & 4 & 0 & 8 \\ 
 B_{e_{ 60 }} &   0 & 8 & 0 & 11 & 8 & 14 & 5 & 14 & 0 & 7 & 7 & 9 & 10 & 0 & 5 & 11 & 8 & 4 & 10 & 14 & 9 & 7 & 8 & 7 & 7 & 9 & 6 & 13 & 3 & 0 & 6 & 1 & 0 & 15 & 10 & 5 & 9 & 1 & 4 & 3 & 4 & 10 & 15 & 8 & 15 & 10 & 7 & 0 & 1 & 2 & 1 & 8 & 15 & 15 & 2 & 5 & 7 & 6 & 8 & 0 
\end{array} $}
}
};
\end{tikzpicture}
\caption{$\mathfrak{B}_\circ$ of an operation $\circ$ which linearizes the map conjugated to PRESENT's mixing layer}\label{rendomopPresent}
\end{table}
\newpage
\section{Conclusions}
Continuing the study of \cite{calderini2017elementary}, here we focused on the class of hidden sums such that $T_\circ\subseteq\AGL(V,+)$ and $T_+\subseteq\AGL(V,\circ)$, which we called practical hidden sums, as these could be used to exploit new attacks on block ciphers.

We gave a lower bound on the number of the practical hidden sums. Then we compared this lower bound with the upper bound given in \cite{calderini2017elementary}.

In the second part, we dealt with the problem of individuating possible practical hidden sums for a given linear map, providing Algorithm \ref{searchop} and showing an example on the case of the PRESENT's mixing layer.

A dual approach is under investigation in \cite{Rob2} ({some of the results are reported in \cite{Rob}}), where the author takes into consideration some S-boxes, which would be strong according to classical cryptanalysis, and identify
some practical hidden sums that weaken significantly the non-linearity properties of the S-Boxes. Once the candidate hidden sums have been found, some (linear) mixing layers are constructed such that they are linear with respect to a hidden sum and the resulting cipher turns out to be attackable by the { clever} attacker.
\section*{Acknowledgements}

The authors would like to thank the anonymous referees for their stimulating suggestions, and several people for interesting discussions: R. Civino and R. Aragona.\\
Part of the results in this paper are from the first's author Master thesis, supervised by the last author. They have been presented partially at WCC 2017.\\
This research was partially funded by the Italian Ministry of Education, Universities and Research, with the project PRIN 2015TW9LSR  ``Group theory and applications''.

\end{document}